\documentclass[12pt]{article}

\usepackage{float}
\usepackage{amsmath, amssymb, amsthm}
\usepackage{mathtools}
\usepackage[margin=1in]{geometry}
\usepackage{hyperref}
\usepackage{enumitem}
\usepackage{comment}
\usepackage{graphicx}
\usepackage{titlesec}
\usepackage[utf8]{inputenc}
\usepackage{authblk}
\usepackage{booktabs}
\usepackage{subcaption}
\usepackage{caption}
\usepackage[
  backend=biber,
  style=numeric-comp,
  sorting=nyt,
  giveninits=true,
  maxnames=99,
  minnames=1,
  doi=false,
  url=false,
  isbn=false,
  sortcites=true
]{biblatex}

\DeclareNameAlias{default}{family-given}

\DeclareFieldFormat[article, incollection, inproceedings, phdthesis, book]{title}{#1}
\DeclareFieldFormat{booktitle}{#1}

\renewbibmacro*{in:}{}

\DeclareFieldFormat{journaltitle}{#1}

\DeclareFieldFormat[article]{volume}{#1}
\DeclareFieldFormat[article]{number}{(#1)}

\DeclareFieldFormat[article]{pages}{pp.\addspace#1}

\renewbibmacro*{date}{\printtext[parens]{\printfield{year}}}

\DeclareBibliographyDriver{article}{%
  \printnames{author}%
  \addcolon\space
  \printfield{title}%
  \addperiod\space
  \printfield{journaltitle}%
  \addcomma\space
  \printfield{volume}%
  \printfield{number}%
  \addcomma\space
  \printfield{pages}%
  \addcomma\space
  \printtext[parens]{\printfield{year}}%
  \addperiod
  \finentry
}

\DeclareBibliographyDriver{phdthesis}{%
  \printnames{author}%
  \addcolon\space
  \printfield{title}%
  \addperiod\space
  \printfield{type}%
  \addcomma\space
  \printfield{institution}%
  \addcomma\space
  \printtext[parens]{\printfield{year}}%
  \addperiod
  \finentry
}

\titleformat{\section}
  {\normalfont\large\bfseries} 
  {\thesection}                
  {1em}                        
  {}

\titleformat{\subsection}
  {\normalfont\large\bfseries}          
  {\thesubsection}             
  {1em}                        
  {}

\titleformat{\subsubsection}
  {\normalfont\normalsize}          
  {\thesubsubsection}             
  {1em}                        
  {}

\newtheoremstyle{normalhead_nobrackets}
  {3pt}        
  {3pt}        
  {\itshape}   
  {}           
  {\bfseries}
  {.}          
  {.5em}       
  {\thmname{#1}\thmnumber{ #2}} 

\theoremstyle{normalhead_nobrackets}

\newtheorem{theorem}{Theorem}[section]
\newtheorem{definition}{Definition}[section]

\newtheorem{corollary}[theorem]{Corollary}
\newtheorem{proposition}[theorem]{Proposition}

\newtheoremstyle{exampstyle}
  {3pt}           
  {3pt}           
  {\normalfont}   
  {}              
  {\bfseries}     
  {.}             
  {.5em}          
  {}              

\theoremstyle{exampstyle}
\newtheorem{example}{Example}[section]

\renewenvironment{abstract}
 {\begin{center}
    \textbf{Abstract} 
  \end{center}
  \list{}{\listparindent 1.5em%
  \itemindent    \listparindent
  \leftmargin    0pt%
  \rightmargin   0pt%
  \parsep        0pt plus 1pt}%
  \item\relax}
 {\endlist}

\title{\textbf{Combinatorial metaplexes and centrality indices for identifying higher-order interactions}}

\author[1]{Hiren J. Dhameliya}
\author[2]{Udit Raj}
\author[1*]{Sudeepto Bhattacharya}

\affil[1]{\footnotesize{Department of Mathematics, School of Natural Sciences, Shiv Nadar Institution of Eminence (Deemed to be University), Tehsil: Dadri, Gautam Buddha Nagar, 201314, Uttar Pradesh, India.}}

\affil[2]{\footnotesize{Department of Electrical Engineering, Indian Institute of Technology, Kanpur, 208016, Uttar Pradesh, India.}}

\affil[ ]{{\footnotesize{$^*$ Corresponding author: sudeepto.bhattacharya@snu.edu.in}}}

\date{}

\begin{document}

\maketitle

\begin{abstract}
\noindent Complex systems consist of interacting units whose interactions may be pairwise, involving two units, or higher-order, involving more than two units simultaneously. Graphs capture pairwise interactions and represent such systems as networks, whereas simplicial complexes can capture higher-order interactions (HoIs) and represent them as higher-order networks comprising simplices. In the clique complex construction, HoIs arise whenever vertices form a clique in the underlying graph.
In classical graph-theoretic and simplicial-complex models, vertices are treated as structurally indistinguishable objects. However, in many real-world systems vertices possess internal structure, and their intrinsic properties influence the HoIs present in the system.
To address this limitation, we introduce the combinatorial metaplex, consisting of two interacting components: an underlying simplicial complex that serves as an admissibility structure specifying boundary-compatible higher-order simplex candidates, and a concentration layer defined by a concentration map assigning a value to each vertex and extending the map to simplices so that the resulting distribution satisfies a conservation relation between vertex weights and facet weights. This concentration layer provides a deterministic threshold rule governing the inclusion of true HoIs. 
Using facet-mediated adjacency and weighted walks, we define one-parameter families of degree, closeness, and harmonic centralities for non-facet simplices, interpolating between those determined solely by the simplicial complex and those determined by concentration-induced coupling.
The framework is illustrated through a representative example, including a comparison of HoIs obtained from the clique complex and the combinatorial metaplex, followed by an analysis of edge centralities within the combinatorial metaplex.
\\

\noindent
\textbf{Keywords:} Combinatorial metaplex, Simplicial complex, Higher-order networks, Facet-mediated adjacency, Weighted walks, Centrality indices.
\\
\textbf{MSC:} 05E45; 55U10; 05C82.
\end{abstract}

\newpage
\section{Introduction}
In the study of complex systems, which are characterised by many interacting units, understanding the structure of interactions among these units is fundamental. Such systems often exhibit collective phenomena such as nonlinearity, emergence, and adaptive behaviour \cite{estrada2024complex, raj2024structure, zhao2021simplicial, boccaletti2006complex}. When interactions occur pairwise, these systems are often modelled using graph-theoretic representations, leading to complex networks in which units are represented as vertices and their interactions as edges \cite{newman2018networks, estrada2015first, upadhyay2019network, battiston2021physics, Bianconi_2021}. Examples of such networks include protein–protein interaction networks, brain networks, and ecological networks \cite{newman2018networks}.

However, many complex systems exhibit interactions among the units of group size greater than two, in addition to pairwise interactions. Such interactions are referred to as higher-order interactions (HoIs) \cite{alvarez2021evolutionary, battiston2020networks, boccaletti2023structure, battiston2022higher}. To represent HoIs explicitly, more general combinatorial structures such as hypergraphs and simplicial complexes are commonly employed \cite{salnikov2019simplicial, aktas2019persistence, torres2020simplicial}. A hypergraph consists of a finite set of vertices together with a collection of subsets of these vertices, called hyperedges. Unlike a graph, where each edge contains exactly two vertices, a hyperedge may contain any number of vertices. A simplicial complex is a collection of subsets of a vertex set that is closed under inclusion: whenever a set belongs to the collection, all of its subsets also belong to it. The elements of a simplicial complex are called simplices, and a simplex containing $k+1$ vertices is called a $k$-simplex. In this representation, HoIs are modelled as simplices (such as triangles, tetrahedra, and higher-dimensional analogues), with the closure property ensuring that all lower-dimensional faces of an interaction are also present \cite{kozlov2008combinatorial}.

In many existing studies, mathematical representations such as graphs, hypergraphs, and simplicial complexes are widely used to model the interaction structure of complex systems \cite{boccaletti2023structure,battiston2021physics,battiston2020networks}. However, these models primarily describe the structure of interactions and disregard intrinsic properties of the individual units forming the system. In many real-world systems, the internal properties of units affect the formation of HoIs \cite{miranda2023simplicial}.

The metaplex framework, introduced by Estrada and collaborators, unifies internal structure and interaction structure within a single representation \cite{estrada2020metaplex, miranda2023simplicial}. In this setting, each vertex carries an internal structure that supports a continuous dynamical process, while interactions between vertices are modelled by a network. This enables the study of diffusion processes that evolve both within components and across their connections.

In some systems, however, intrinsic properties of units are discrete and quantifiable rather than continuous \cite{bairey2016high, badalyan2024structure, nakajima2025inference, contisciani2025flexible}. Vertices may carry measurable attributes such as the 
$h$-index or citation count of an author in a co-authorship network, the population of a city in a transportation network, or the abundance of a species in an ecological corridor network \cite{singh2020node, shanu2019graph, contisciani2025flexible}. These heterogeneous quantities influence interaction structure but are not generated by internal dynamical processes \cite{miranda2023simplicial}. Studies on vertex-weighted networks demonstrate that such intrinsic vertex weights affect structural measures including centrality and flow \cite{singh2020node}.

Many higher-order network models employ abstract simplicial complexes, in which higher-dimensional simplices arise purely from combinatorial inclusion \cite{zhao2021simplicial}. In particular, in clique-based constructions every complete subgraph is interpreted as a HoI, regardless of vertex attributes \cite{raj2024structure, giusti2016two}. Consequently, HoIs emerge solely from pairwise closure, without incorporating the intrinsic properties of the participating units.

This motivates the introduction of the combinatorial metaplex. In this framework, each vertex is endowed with a combinatorial structure representing the internal structure of the corresponding unit. A concentration is assigned to each vertex as a scalar quantity representing an intrinsic quantitative property of the unit. These vertex concentrations constitute the fundamental numerical quantities in the model.

Unlike multilayer network models \cite{boccaletti2014structure,kivela2014multilayer}, which extend networks outward through multiple layers and interlayer edges connecting vertices across these layers, the combinatorial metaplex introduces vertex-level heterogeneity by assigning internal structure to individual units.

Within this vertex-driven setting, higher-dimensional simplices are not assigned primary weights. Instead, their weights are induced from the concentrations of their constituent vertices through the simplicial inclusion structure, implemented via the fractional weight distribution principle \cite{batagelj2020fractional,perianes2016constructing}. Because a vertex may participate in multiple simplices, its concentration is distributed among incident simplices (that is, simplices containing that vertex) according to this inclusion relation. Consequently, any numerical value associated with a simplex arises solely from aggregated vertex contributions, and no independent weight is assigned at the simplex level. This distinguishes the combinatorial metaplex from weighted simplicial complexes (see, e.g., \cite{sharma2017weighted,carstens2013persistent,baccini2022weighted}), where a non-negative real-value is assigned directly to simplices, thereby prescribing numerical strengths to HoIs.

In the combinatorial metaplex, the inclusion of a higher-order simplex is governed by a structural–quantitative condition. A simplex is included only if the aggregated concentration of its boundary simplices exceeds a prescribed threshold. Thus, simplices arise from two simultaneous requirements: combinatorial admissibility and sufficient inherited concentration. This ensures that realised true HoIs are supported by intrinsic vertex-level properties rather than appearing solely through combinatorial closure.

Finally, we introduce centrality measures on combinatorial metaplexes—including degree, closeness, and harmonic centralities—to analyse how intrinsic vertex-level properties influence the structure of HoIs.

A simplicial complex of dimension $1$ (or the $1$-skeleton of a simplicial complex $\Delta$) is a graph, and various centrality indices for pairwise networks have been introduced in the literature \cite{upadhyay2017network, upadhyay2019network, opsahla2010node, scardoni2012centralities}. The most basic centrality, degree centrality, counts the number of incident edges to a vertex (that is, edges containing that vertex). In an edge-weighted graph, the strength of a vertex is defined as the aggregate weight of its incident edges.  In classical graph-theoretic formulations, centrality measures have primarily focused on the weighted edge-based adjacency structure. The incorporation of vertex weights into centrality definitions has been explored in \cite{abbasi2013hybrid, singh2020node, akanmu2014weighted, kapoor2013weighted}. Our work extends these constructions by defining centrality measures for fully weighted graphs, in which both vertices and edges are assigned weights, thereby providing a framework that couples intrinsic vertex properties with adjacency-based network structure.

Centrality measures in higher-order networks quantify the importance of simplices within the interaction structure. Extending graph-based centralities to simplicial complexes requires an appropriate notion of adjacency between simplices. The adjacency between two distinct $k$-simplices has been defined in terms of upper adjacency (when there exists a $(k+1)$-simplex containing both) and lower adjacency (when there exists a $(k-1)$-simplex contained in both) \cite{goldberg2002}. A generalised notion of adjacency, in which two distinct $k$-simplices are lower adjacent but not upper adjacent, was introduced in \cite{estrada2018}. The adjacency between arbitrary simplices was defined in \cite{serrano2020, serrano2020simplicial}.

The generalised degree centrality in the $d$-exclusive simplicial complex of $\Delta$ (i.e., the family of $d$-simplices of $\Delta$ and their faces) was defined in \cite{courtney2016generalized} and subsequently extended to betweenness and closeness centralities in \cite{raj2023}. Using a tensorial representation of the $d$-exclusive simplicial complex, the strength of a simplex in a weighted simplicial complex was defined in \cite{courtney2017weighted}. In defining centralities for the combinatorial metaplex, we adopt the facet-based definition of adjacency proposed in \cite{raj2025study}, in which two distinct $k$-simplices are said to be adjacent whenever they are contained in a common facet.

The main contributions of this work are as follows: (1) We introduce the combinatorial metaplex, a two-layered construction that couples intrinsic vertex-level data with higher-order simplicial structure. (2) We extend the vertex concentration map to the entire simplicial complex via a fractional contribution map, propagating intrinsic vertex properties to higher-dimensional simplices through the incidence structure. (3) We establish a conservation relation between vertex concentrations and maximal higher-order simplices: the total induced weight of all facets equals the total concentration assigned to the $0$-simplices. (4) We develop a concentration-based framework for constructing simplicial complexes from pairwise data. Unlike clique-based constructions, which consider every clique of size greater than two as a HoI, our approach identifies only those simplices whose aggregated boundary concentration exceeds a prescribed threshold, thereby detecting concentration-supported HoIs. (5) We define degree centrality for fully weighted graphs consistent with previous definitions in weighted graphs. (6) We introduce degree, closeness, and harmonic centralities for the combinatorial metaplex using a facet-mediated notion of adjacency and weighted walks. (7) We illustrate the proposed framework using a real-world network example. This example demonstrates the construction of the combinatorial metaplex and shows how the concentration-based inclusion rule identifies true HoIs, in contrast to the purely combinatorial HoIs obtained from the clique complex.

Section \ref{1sec:preliminaries} presents the standard definitions from combinatorial algebraic topology required for completeness. These notions are well established in the literature \cite{bondy2008graph, hatcher_algebraic_2001, kozlov2008combinatorial, raj2023, raj2024structure}.

\section{Preliminaries}\label{1sec:preliminaries}

Let \( V = \{v_0, v_1, \ldots, v_n\} \) be a finite, non-empty set of vertices, where \( n \in \mathbb{N} \). We denote the power set of $V$ as $\mathcal{P}(V)$.

\begin{definition}
A graph $G$ is a three-tuple \( G = (V, E, \psi_G) \), where:
\begin{enumerate}
    \item[(i)] \( V \) is the set of vertices,
    \item[(ii)] \( E \) is the set of edges,
    \item[(iii)] \( \psi_G : E \rightarrow V \cup [V]^2 \) is the incidence map that assigns to each edge either one or two vertices (its endpoints), where
    \( [V]^2 = \{ X \in \mathcal{P}(V) \mid |X| = 2 \} \).
\end{enumerate}
\end{definition}

\begin{definition}
A simple graph is a graph \( G = (V, E) \), where \( E \subseteq [V]^2 \). That is, each edge is a two-element subset of \(V\); in particular, there are no loops or multiple edges.
\end{definition}

Throughout this work, unless otherwise specified, the term graph refers exclusively to a simple graph.

\begin{definition}
Let $G = (V,E)$ be a graph and let $Z \subseteq V$. 
The induced subgraph of $G$ on $Z$, denoted by $G[Z]$, 
is the graph whose vertex set is $Z$ and whose edge set is
\[
E_Z = \big\{ \{u,v\} \in E \;\big|\; u,v \in Z \big\}.
\]
That is, $G[Z]$ contains exactly those edges of $G$ whose endpoints both lie in $Z$.
\end{definition}

\begin{definition}
A simplex over \( V \) is a non-empty subset \( \sigma \subseteq V \).
If the cardinality \( |\sigma| = d + 1 \), the dimension of \( \sigma \) is defined as
\[
dim(\sigma) = |\sigma| - 1 = d.
\]
In this case, \( \sigma \) is called a \textit{\( d \)-simplex}, and is often denoted as \( \sigma^{(d)} \).
\end{definition}

\begin{definition}
Let $\sigma^{(q)}$ be a $q$-simplex over a vertex set $V$. The \text{boundary} of $\sigma^{(q)}$, denoted $\mathcal{B}(\sigma^{(q)})$, is the family of all $(q-1)$-dimensional faces of $\sigma^{(q)}$.

Formally, we define
\[
\mathcal{B}(\sigma^{(q)}) =
\left\{ \tau \subsetneq \sigma^{(q)} \mid dim(\tau)=q-1 \right\}.
\]
\end{definition}

\begin{definition}
If \( \sigma^{(p)} \subseteq \sigma^{(q)} \), then \( \sigma^{(p)} \) is called a \( p \)-dimensional face of \( \sigma^{(q)} \). If \( p < q \), then $\sigma^{(p)}$ is called a proper face.
\end{definition}

\begin{definition}
A collection \( \Delta \) of simplices is called an abstract simplicial complex if it satisfies the condition:
\[
\forall \sigma \in \Delta,\text{ and } (\tau \in \mathcal{P}(V), \text{ if } \tau \subseteq \sigma \Rightarrow \tau \in \Delta).
\]
That is, every face of a simplex in \( \Delta \) must also belong to \( \Delta \).
The vertex set of $\Delta$ is denoted $V(\Delta) = \{ v_i \in V \mid \{v_i\} \in \Delta \}$.
\end{definition}
In this terminology, a $0$-simplex is a vertex, a $1$-simplex is an edge, a $2$-simplex is a triangle, a $3$-simplex is a tetrahedron, and, in general, a $k$-simplex consists of $k+1$ vertices. Throughout this work, the term simplicial complex refers to an abstract simplicial complex.

\begin{definition}
The dimension of a simplicial complex is the maximum number $q \in \mathbb{N}$ such that $\Delta$ contains a $q$-simplex:
\[
dim(\Delta) = \max\{q \in \mathbb{N} \mid \sigma^{(q)} \in \Delta\}.
\]
\end{definition}
 
\begin{definition}
The collection of all simplices in \( \Delta \) of dimension at most \( d \) is called the \( d \)-skeleton of \( \Delta \), and is denoted by:
\[
\Delta^{(d)} = \{\sigma \in \Delta \mid dim(\sigma) \le d\}.
\]
\end{definition}
The 1-skeleton of a simplicial complex \( \Delta \) is the graph formed by its 0- and 1-simplices. The vertices of the graph correspond to the 0-simplices, and an edge exists between two vertices if and only if the corresponding 1-simplex belongs to the 1-skeleton of \( \Delta \).

\begin{definition}
A facet of a simplicial complex $\Delta$ is a simplex that is not contained in any other simplex of $\Delta$. Formally, a simplex $\gamma \in \Delta$ is a facet if
\[
\nexists \tau \in \Delta \setminus \{\gamma\} \text{ such that } \gamma \subseteq \tau
\]
\end{definition}

\begin{definition}
The set of $q$-simplices of $\Delta$ is denoted by $S_q$:
\[
S_q = \{ \sigma \in \Delta \mid dim(\sigma) = q \} \quad \text{with} \quad |S_q| = n_q.
\]
We denote an enumeration of all $q$-dimensional simplices of $\Delta$ as:
\[
S_q = \{ \sigma_1^{(q)}, \sigma_2^{(q)}, \ldots, \sigma_{n_q}^{(q)} \}.
\]
\end{definition}

Next, we define the upper degree following the framework in \cite{goldberg2002}.
\begin{definition}
Let $0 \le q < dim(\Delta)$. The upper degree of a $q$-simplex $\sigma_i^{(q)}$ is defined as the number of $(q+1)$-simplices containing it:
\[
k_{\sigma_i^{(q)}} =
\left|
\left\{
\sigma^{(q+1)} \in S_{q+1} \mid
\sigma_i^{(q)} \subsetneq \sigma^{(q+1)}
\right\}
\right|.
\]
\end{definition}

\begin{definition}
A graph $G = (V,E)$ is said to be complete if
\[
E(G) 
= 
\big\{ \{u,v\} \subseteq V \;\big|\; u \neq v \big\}.
\]
\end{definition}

\begin{definition}
Let $G = (V,E)$ be a graph. 
The clique complex of $G$, denoted by $C(G)$, 
is the abstract simplicial complex defined by
\[
C(G) 
= 
\big\{ \sigma \subseteq V \;\big|\; G[\sigma] \text{ is a complete graph} \big\}.
\]
\end{definition}

\begin{definition}
A metaplex is a $4$-tuple
\[
\mathcal{X} = (V, E, \mathcal{I}, \omega),
\]
where $(V, E)$ is a graph,
\[
\omega = \{\Omega_j\}_{j=1}^k
\]
is a set of locally compact metric spaces $\Omega_j$ with Borel measures $\mu_j$, and
\[
\mathcal{I} : V \to \omega
\]
is an assignment map \cite{estrada2020metaplex}.
\end{definition}

\section{The combinatorial metaplex and the weight map}\label{1sec:CM framework}

\subsection{The combinatorial metaplex}

\bigskip

As mentioned previously, the metaplex framework introduced in \cite{estrada2020metaplex} models the internal vertex structure using locally compact metric–measure spaces. These spaces naturally support continuous dynamics, such as diffusion. However, in many systems, the internal structure of units is inherently discrete and does not admit a meaningful continuous description. In such cases, replacing internal metric–measure spaces with a finite combinatorial structure provides a natural generalisation. To model such systems, we introduce the combinatorial metaplex framework. In this framework, each vertex is associated with a finite combinatorial internal structure endowed with a discrete measure, which is formalised through a concentration map defined on the set of vertices.
\begin{definition}
A combinatorial metaplex (CM) is a quintuple
\begin{equation}\label{1eq:CM}
     \mathcal{K}_c = (V, \Delta, \mathcal{Y}, \mathcal{W}, \mathcal{I}),
\end{equation}
where:
\begin{enumerate}
     \item[(i)] $V = \{v_0, v_1, \dots, v_n\}$ is a finite, non-empty set of vertices.
     \item[(ii)] $\Delta$ is a finite simplicial complex over $V$.
     \item[(iii)] $\mathcal{Y} = \{Y_i\}_{i=0}^n$ is a family of finite, non-empty sets indexed by the vertices.
     \item[(iv)] $\mathcal{W} = \{w_i\}_{i=0}^n$ is a collection of non-negative weight maps
     \[
     w_i : Y_i \to \mathbb{Q} \setminus \mathbb{Q}^-,
     \]
     such that $w_i$ is not identically zero, i.e.,
     \[
         \sum_{x \in Y_i} w_i(x) > 0 \quad \text{for all } 0 \le i \le n.
     \]
    
    \item[(v)] $\mathcal{I} : V \to \mathcal{Y}$ is an assignment map defined by
    \[
        \mathcal{I}(v_i) = Y_i \quad \text{for all } 0 \le i \le n.
    \]
\end{enumerate}

Each $w_i$ induces a discrete measure $\mu_i$ on the power set $\mathcal{P}(Y_i)$, such that for any subset $A \subseteq Y_i$,
\begin{equation*}
    \mu_i(A) = \sum_{x \in A} w_i(x).
\end{equation*}
Thus, each triple $(Y_i, \mathcal{P}(Y_i), \mu_i)$ defines a discrete measure space associated uniquely with the vertex $v_i \in V$ \cite{rudin1987real}.
\end{definition}

The {concentration} at a vertex $v \in V$ is defined as the discrete measure of its internal structure:

\begin{definition}
Given a CM as in \eqref{1eq:CM}, for a vertex $v_j \in V, \ 0 \leq j \leq n$, let $\mathcal{I}(v_j) = Y_j \in \mathcal{Y}$.
The concentration is a map
\[
a : V \to \mathbb{Q}^{+},
\]
defined by
\begin{equation*}
    a(v_j) := \mu_j(Y_j) = \sum_{x \in Y_j} w_j(x).
\end{equation*}
\end{definition}

In the CM framework, vertices are endowed with a combinatorial structure that encodes their intrinsic discrete properties. The concentration map assigns to each vertex the discrete measure of its internal structure, that is, the total weight of the elements in the associated set. Consequently, each vertex is assigned a strictly positive concentration and thus contributes non-trivially.

Henceforth, the term `weight' will be used interchangeably with concentration. In the subsequent sections, we develop a method to extend the weight map from the vertex set to the simplicial complex.

\subsection{The weight map on 0-simplices}

The set of vertices $V$ constitutes the primitive objects of the simplicial complex. In the classical theory of simplicial complexes, these vertices are treated as structurally indistinguishable. In the CM, however, this homogeneity is broken: vertices are differentiated through the assignment of positive rational values, representing their internal properties, via the concentration map. 

Let $\Delta$ be a simplicial complex over a vertex set $V$. We define the weight map as follows:
\[
a : V \to \mathbb{Q}^+, \qquad v \mapsto a(v).
\]
Restricting the map $a$ to the active vertex set of the complex $\Delta$, we obtain the following:
\[
a|_{V(\Delta)} : V(\Delta) \to \mathbb{Q}^+.
\]
This restriction naturally induces a map on the set of $0$-simplices (vertices) of $\Delta$. Denoting this map by $a'$, we have:
\[
a' : S_0 \to \mathbb{Q}^+, \qquad 
\{v_i\} \mapsto a'(\{v_i\}) = a|_{V(\Delta)}(v_i).
\]
For convenience, we denote the map $a'$ again by $a$.
With the vertices now distinguished by their concentration values, we next define how this weight extends to higher-dimensional simplices (edges, triangles, etc.), thereby yielding a fully weighted simplicial structure.

\subsection{The fractional weight contribution map and contribution numbers}

In this section, we describe a constructive inductive procedure for extending the weight map dimension by dimension over the simplicial complex. The weights are initially assigned to the vertices, and the weight of a $q$-simplex is determined by the weights of its boundary simplices.

Let the weight map on the set of $0$-simplices, $S_0$, be
\[
a : S_0 \to \mathbb{Q}^{+}.
\]

We begin by defining the fractional weight contribution map $a_1$, which assigns to each pair of a $0$-simplex and a $1$-simplex a non-negative rational number representing the fraction of the weight of the vertex that contributes to the edge.

\begin{definition}
    The fractional weight contribution map $a_1$ is defined as
\[
a_1 : S_0 \times S_1 \longrightarrow \mathbb{Q}\setminus \mathbb{Q}^-,
\]
\[
(\sigma^{(0)}, \sigma^{(1)}) \longmapsto a_1(\sigma^{(0)}, \sigma^{(1)}),
\]
satisfying the following conditions:

\begin{itemize}
    \item[(i)] If $\sigma^{(0)} \not\subset \sigma^{(1)}$, then 
    \(
    a_1(\sigma^{(0)}, \sigma^{(1)}) = 0.
    \)
    
    \item[(ii)] If $\sigma^{(0)} \subset \sigma^{(1)}$, then
    \(
    0 < a_1(\sigma^{(0)}, \sigma^{(1)}) \le 1.
    \)
    
    \item[(iii)] For each $\sigma^{(0)} \in S_0$ with $k_{\sigma^{(0)}} \neq 0$,
    \(
    \sum_{\sigma^{(1)} \in S_1}
    a_1(\sigma^{(0)}, \sigma^{(1)}) = 1.
    \)
\end{itemize}
\end{definition}

From properties (i) and (ii), it follows that
\[
0 \le a_1(\sigma^{(0)}, \sigma^{(1)}) \le 1
\quad
\text{for all } \sigma^{(0)} \in S_0,\ \sigma^{(1)} \in S_1.
\]
If $k_{\sigma^{(0)}} = 0$, then by property (i),
\[
a_1(\sigma^{(0)}, \sigma^{(1)}) = 0
\quad
\text{for all } \sigma^{(1)} \in S_1.
\]

Based on this fractional map, we associate to each pair 
$(\sigma^{(0)}, \sigma^{(1)})$ a contribution number 
$n^{\sigma^{(1)}}_{\sigma^{(0)}}$, defined by
\[
n^{\sigma^{(1)}}_{\sigma^{(0)}}
=
a_1(\sigma^{(0)}, \sigma^{(1)})\, a(\sigma^{(0)}).
\]
This quantity represents the absolute weight contribution of the vertex $\sigma^{(0)}$ to the edge $\sigma^{(1)}$. Suppose that a $0$-simplex $\sigma^{(0)}$ is contained in 
$k_{\sigma^{(0)}} \neq 0$ distinct $1$-simplices
\[
\sigma^{(1)}_{i_1},\ \sigma^{(1)}_{i_2},\ \ldots,\ 
\sigma^{(1)}_{i_{k_{\sigma^{(0)}}}},
\quad
1 \le i_j \le n_1,\ 
1 \le j \le k_{\sigma^{(0)}}.
\]
Then, the weight $a(\sigma^{(0)})$ is partitioned among these 
$k_{\sigma^{(0)}}$ simplices, and by property (ii),
\[
0 < n^{\sigma^{(1)}_{i_j}}_{\sigma^{(0)}} \le a(\sigma^{(0)}),
\quad
1 \le j \le k_{\sigma^{(0)}}.
\]
Moreover, by property (iii), the conservation property holds:
\[
\sum_{j=1}^{k_{\sigma^{(0)}}}
n^{\sigma^{(1)}_{i_j}}_{\sigma^{(0)}}
=
a(\sigma^{(0)}).
\]

For $1$-simplices not containing $\sigma^{(0)}$, or when 
$k_{\sigma^{(0)}} = 0$, property (i) implies that
\[
n^{\sigma^{(1)}}_{\sigma^{(0)}} = 0.
\]
Consequently,
\[
\sum_{\sigma^{(1)} \in S_1}
n^{\sigma^{(1)}}_{\sigma^{(0)}}
=
a(\sigma^{(0)}),
\]
so the total weight of each vertex is conserved.

For convenience, we also denote the extended map on $S_1$ again by $a$. Finally, the weight of a $1$-simplex $\sigma^{(1)} \in S_1$ is defined as the sum of the contributions of its constituent vertices. Thus, extending the weight map $a$ to $S_1$, we obtain
\[
a(\sigma^{(1)})
=
\sum_{i=1}^{n_0}
n_{\sigma^{(0)}_i}^{\sigma^{(1)}}
=
\sum_{i=1}^{n_0}
a_1(\sigma^{(0)}_i, \sigma^{(1)}) \, a(\sigma^{(0)}_i),
\]
by property $(ii)$, $a(\sigma^{(1)}) > 0, \ \forall \sigma^{(1)} \in S_1$.

\subsection{The general fractional weight contribution map and higher-order extensions}
We generalise the previous construction to obtain an inductive procedure for extending the weight map from $(q-1)$-simplices to $q$-simplices.

To define the extension of the weight map from $S_{q-1}$ to $S_q$, we introduce the fractional weight contribution map $a_q$, which determines how the weight of a $(q-1)$-simplex is distributed among the $q$-simplices that contain it.

The value $a_q(\sigma^{(q-1)}, \sigma^{(q)})$ represents the fraction of the weight of the $(q-1)$-simplex $\sigma^{(q-1)}$ contributed to the weight of the $q$-simplex $\sigma^{(q)}$.

\begin{definition}
Let $1 \le q \le dim(\Delta)$. The fractional weight contribution map $a_q$ is defined as
\[
a_q : S_{q-1} \times S_q \longrightarrow \mathbb{Q} \setminus \mathbb{Q}^-,
\]
\[
(\sigma^{(q-1)}, \sigma^{(q)}) \longmapsto a_q(\sigma^{(q-1)}, \sigma^{(q)}),
\]
satisfying the following conditions:

\begin{itemize}
    \item[(i)]\label{1property_a_q_1} If $\sigma^{(q-1)} \not\subset \sigma^{(q)}$, then
    \(
    a_q(\sigma^{(q-1)}, \sigma^{(q)}) = 0.
    \)

    \item[(ii)]\label{1property_a_q_2} If $\sigma^{(q-1)} \subset \sigma^{(q)}$, then
    \(
    0 < a_q(\sigma^{(q-1)}, \sigma^{(q)}) \le 1.
    \)

    \item[(iii)]\label{1property_a_q_3} For each $\sigma^{(q-1)} \in S_{q-1}$ with $k_{\sigma^{(q-1)}} \neq 0$,
    \(
    \sum_{\sigma^{(q)} \in S_q}
    a_q(\sigma^{(q-1)}, \sigma^{(q)}) = 1.
    \)
\end{itemize}
\end{definition}

From (i) and (ii), we obtain
\[
0 \le a_q(\sigma^{(q-1)}, \sigma^{(q)}) \le 1
\quad
\text{for all } \sigma^{(q-1)} \in S_{q-1},\ \sigma^{(q)} \in S_q.
\]

If $k_{\sigma^{(q-1)}} = 0$, then by property (i),
    \[
    a_q(\sigma^{(q-1)}, \sigma^{(q)}) = 0
    \quad
    \text{for all } \sigma^{(q)} \in S_q.
    \]

Based on this map, we define a contribution number
\[
n^{\sigma^{(q)}}_{\sigma^{(q-1)}}
=
a_q(\sigma^{(q-1)}, \sigma^{(q)}) \, a(\sigma^{(q-1)}),
\]
which represents the absolute weight contribution of the $(q-1)$-simplex $\sigma^{(q-1)}$ to the $q$-simplex $\sigma^{(q)}$.

Suppose that a $(q-1)$-simplex $\sigma^{(q-1)}$ is contained in
$k_{\sigma^{(q-1)}} \neq 0$ distinct $q$-simplices
\[
\sigma^{(q)}_{i_1},\ \sigma^{(q)}_{i_2},\ \ldots,\ 
\sigma^{(q)}_{i_{k_{\sigma^{(q-1)}}}}
\quad
1 \le i_j \le n_{q},\ 
1 \le j \le k_{\sigma^{(q-1)}}.
\]
Then the weight $a(\sigma^{(q-1)})$ is partitioned among these simplices, and by property (ii),
\[
0 < n^{\sigma^{(q)}_{i_j}}_{\sigma^{(q-1)}} \le a(\sigma^{(q-1)}),
\quad
1 \le j \le k_{\sigma^{(q-1)}},
\]
with the conservation property
\[
\sum_{j=1}^{k_{\sigma^{(q-1)}}}
n^{\sigma^{(q)}_{i_j}}_{\sigma^{(q-1)}}
=
a(\sigma^{(q-1)}).
\]

For $q$-simplices not containing $\sigma^{(q-1)}$, or when 
$k_{\sigma^{(q-1)}} = 0$, we get
\[
n^{\sigma^{(q)}}_{\sigma^{(q-1)}} = 0.
\]
Consequently,
\[
\sum_{\sigma^{(q)} \in S_q}
n^{\sigma^{(q)}}_{\sigma^{(q-1)}}
=
a(\sigma^{(q-1)}),
\]
so the total weight of each $(q-1)$-simplex is conserved.

Finally, the weight of a $q$-simplex $\sigma^{(q)} \in S_q$ is defined as the sum of the contributions of its boundary simplices:
\[
a(\sigma^{(q)})
=
\sum_{\sigma^{(q-1)} \in S_{q-1}}
n^{\sigma^{(q)}}_{\sigma^{(q-1)}}
=
\sum_{\sigma^{(q-1)} \in S_{q-1}}
a_q(\sigma^{(q-1)}, \sigma^{(q)}) \, a(\sigma^{(q-1)}),
\]
by property $(ii)$, $a(\sigma^{(q)}) > 0, \ \forall \sigma^{(q)} \in S_q$.

Thus, inductively, the concentration map $a$ extends from $S_0$ to all simplices of $\Delta$, yielding a well-defined weight map
\[
a : \Delta \longrightarrow \mathbb{Q} \setminus \mathbb{Q}^-.
\]

We now define the composition of two consecutive fractional weight contribution maps.

\subsection{The composition of the fractional weight contribution maps}
The composition of two consecutive fractional weight contribution maps is again a fractional weight contribution map, describing the fraction of the weight of a $(q-1)$-simplex contributed to a $(q+1)$-simplex. This composed map enables the computation of weight contributions across multiple dimensions, in particular from vertices to higher-dimensional simplices.

The composition map is useful for determining the fractional contribution number of a $0$-simplex to any $q$-simplex.

To derive this composition map, we start with the definition of $a(\sigma^{(q+1)})$ and expand it iteratively:
\begin{align*}
a(\sigma^{(q+1)}) &=
\sum_{i=1}^{n_q}
a_{q+1}\bigl(\sigma_i^{(q)}, \sigma^{(q+1)}\bigr)
\, a(\sigma_i^{(q)}) \\
\intertext{Substituting the value of $a(\sigma_i^{(q)})$ from the previous dimension:}
&=
\sum_{i=1}^{n_q}
a_{q+1}\bigl(\sigma_i^{(q)}, \sigma^{(q+1)}\bigr)
\left(
\sum_{j=1}^{n_{q-1}}
a_q\bigl(\sigma_j^{(q-1)}, \sigma_i^{(q)}\bigr)
\, a(\sigma_j^{(q-1)})
\right) \\
&=
\sum_{i=1}^{n_q}
\sum_{j=1}^{n_{q-1}}
a_{q+1}\bigl(\sigma_i^{(q)}, \sigma^{(q+1)}\bigr)
\, a_q\bigl(\sigma_j^{(q-1)}, \sigma_i^{(q)}\bigr)
\, a(\sigma_j^{(q-1)}) \\
\intertext{Interchanging the order of summation (valid as all summand terms are positive):}
&=
\sum_{j=1}^{n_{q-1}}
\left(
\sum_{i=1}^{n_q}
a_{q+1}\bigl(\sigma_i^{(q)}, \sigma^{(q+1)}\bigr)
\, a_q\bigl(\sigma_j^{(q-1)}, \sigma_i^{(q)}\bigr)
\right)
a(\sigma_j^{(q-1)}).
\end{align*}

Thus we may define,
\[
a(\sigma^{(q+1)}) =
\sum_{j=1}^{n_{q-1}}
(a_{q+1} \circ a_q)\bigl(\sigma_j^{(q-1)}, \sigma^{(q+1)}\bigr)
\, a(\sigma_j^{(q-1)}),
\]
where
\[
(a_{q+1} \circ a_q)\bigl(\sigma^{(q-1)}, \sigma^{(q+1)}\bigr)
=
\sum_{\sigma^{(q)} \in S_q}
a_{q+1}\bigl(\sigma^{(q)}, \sigma^{(q+1)}\bigr)
\, a_q\bigl(\sigma^{(q-1)}, \sigma^{(q)}\bigr).
\]

\begin{theorem}
Let $1 \le q \le \dim(\Delta)-1$. 
Let $a_q$ and $a_{q+1}$ be the fractional weight contribution maps. 
Then their composition map is
\[
a_{q+1} \circ a_q : S_{q-1} \times S_{q+1} 
\longrightarrow \mathbb{Q} \setminus \mathbb{Q}^-,
\]
\[
(\sigma^{(q-1)}, \sigma^{(q+1)}) 
\longmapsto 
(a_{q+1} \circ a_q)(\sigma^{(q-1)}, \sigma^{(q+1)}),
\]
where
\[
(a_{q+1} \circ a_q)(\sigma^{(q-1)}, \sigma^{(q+1)})
=
\sum_{\sigma^{(q)} \in S_q}
a_{q+1}(\sigma^{(q)}, \sigma^{(q+1)})
a_q(\sigma^{(q-1)}, \sigma^{(q)}).
\]

The map $a_{q+1} \circ a_q$ satisfies the following conditions:

\begin{itemize}
    \item[(i)] If $\sigma^{(q-1)} \not\subset \sigma^{(q+1)}$, then
    \(
    (a_{q+1} \circ a_q)(\sigma^{(q-1)}, \sigma^{(q+1)}) = 0.
    \)

    \item[(ii)] If $\sigma^{(q-1)} \subset \sigma^{(q+1)}$, then
    \(
    0 < (a_{q+1} \circ a_q)(\sigma^{(q-1)}, \sigma^{(q+1)}) \le 1.
    \)

    \item[(iii)] For each $\sigma^{(q-1)} \in S_{q-1}$ if $\exists \sigma^{(q+1)} \in S_{q+1}$ such that $\sigma^{(q-1)} \subset \sigma^{(q+1)}$, then
    \[
    \sum_{\sigma^{(q+1)} \in S_{q+1}}
    (a_{q+1} \circ a_q)(\sigma^{(q-1)}, \sigma^{(q+1)}) = 1.
    \]
\end{itemize}
\end{theorem}

\begin{proof}

Since
\[
a_q(\sigma^{(q-1)}, \sigma^{(q)}) 
\in \mathbb{Q} \setminus \mathbb{Q}^-,
\qquad
a_{q+1}(\sigma^{(q)}, \sigma^{(q+1)}) 
\in \mathbb{Q} \setminus \mathbb{Q}^-,
\]
their product belongs to $\mathbb{Q} \setminus \mathbb{Q}^-$, and hence
\[
(a_{q+1} \circ a_q)(\sigma^{(q-1)}, \sigma^{(q+1)})
\in \mathbb{Q} \setminus \mathbb{Q}^-.
\]

\medskip

\noindent
\text{Proof of (i):}
If $\sigma^{(q-1)} \not\subset \sigma^{(q+1)}$, then there does not exist 
$\sigma^{(q)} \in S_q$ such that
\[
\sigma^{(q-1)} \subset \sigma^{(q)} 
\subset \sigma^{(q+1)}.
\]
Hence for every $\sigma^{(q)} \in S_q$, either
\[
a_q(\sigma^{(q-1)}, \sigma^{(q)}) = 0
\quad \text{or} \quad
a_{q+1}(\sigma^{(q)}, \sigma^{(q+1)}) = 0.
\]
Therefore every summand in
\[
(a_{q+1} \circ a_q)(\sigma^{(q-1)}, \sigma^{(q+1)})
=
\sum_{\sigma^{(q)} \in S_q}
a_{q+1}(\sigma^{(q)}, \sigma^{(q+1)})
a_q(\sigma^{(q-1)}, \sigma^{(q)})
\]
is zero, and thus
\[
(a_{q+1} \circ a_q)(\sigma^{(q-1)}, \sigma^{(q+1)}) = 0.
\]

\medskip

\noindent
\text{Proof of (ii):}
If $\sigma^{(q-1)} \subset \sigma^{(q+1)}$, then there exists at least one 
$\sigma^{(q)} \in S_q$ such that
\[
\sigma^{(q-1)} \subset \sigma^{(q)} 
\subset \sigma^{(q+1)}.
\]
For such $\sigma^{(q)}$,
\[
a_q(\sigma^{(q-1)}, \sigma^{(q)}) > 0
\quad \text{and} \quad
a_{q+1}(\sigma^{(q)}, \sigma^{(q+1)}) > 0.
\]
Hence at least one summand in the defining sum is strictly positive,
and therefore
\[
0 < (a_{q+1} \circ a_q)(\sigma^{(q-1)}, \sigma^{(q+1)}).
\]
Since 
\[
0 \le a_{q+1}(\sigma^{(q)}, \sigma^{(q+1)}) \le 1
\quad \text{and} \quad
a_q(\sigma^{(q-1)}, \sigma^{(q)}) \ge 0,
\]
we have
\[
a_{q+1}(\sigma^{(q)}, \sigma^{(q+1)})
\, a_q(\sigma^{(q-1)}, \sigma^{(q)})
\le
a_q(\sigma^{(q-1)}, \sigma^{(q)}).
\]
Summing over $\sigma^{(q)} \in S_q$ gives
\begin{align*}
(a_{q+1} \circ a_q)(\sigma^{(q-1)}, \sigma^{(q+1)})
&=
\sum_{\sigma^{(q)} \in S_q}
a_{q+1}(\sigma^{(q)}, \sigma^{(q+1)})
a_q(\sigma^{(q-1)}, \sigma^{(q)}) \\
&\le
\sum_{\sigma^{(q)} \in S_q}
a_q(\sigma^{(q-1)}, \sigma^{(q)}) \\
&= 1.
\end{align*}

\medskip

\noindent
\text{Proof of (iii):}
Let $\sigma^{(q-1)} \in S_{q-1}$ such that there exists
$\sigma^{(q+1)} \in S_{q+1}$ with
$\sigma^{(q-1)} \subset \sigma^{(q+1)}$.
Then
\begin{align*}
\sum_{\sigma^{(q+1)} \in S_{q+1}}
(a_{q+1} \circ a_q)(\sigma^{(q-1)}, \sigma^{(q+1)})
&=
\sum_{\sigma^{(q+1)} \in S_{q+1}}
\sum_{\sigma^{(q)} \in S_q}
a_{q+1}(\sigma^{(q)}, \sigma^{(q+1)})
a_q(\sigma^{(q-1)}, \sigma^{(q)}) \notag \\
&=
\sum_{\sigma^{(q)} \in S_q}
\left(
\sum_{\sigma^{(q+1)} \in S_{q+1}}
a_{q+1}(\sigma^{(q)}, \sigma^{(q+1)})
\right)
a_q(\sigma^{(q-1)}, \sigma^{(q)}) \notag \\
&=
\sum_{\sigma^{(q)} \in S_q}
a_q(\sigma^{(q-1)}, \sigma^{(q)}) \notag \\
&= 1.
\end{align*}
\end{proof}

Let $a : S_{q-1} \to \mathbb{Q} \setminus \mathbb{Q}^-$ 
be a weight map on $S_{q-1}$, and let 
$a_{q+1} \circ a_q$ be the composition of the fractional weight contribution maps.

Then the weight map extends to $S_{q+1}$ by defining
\[
a : S_{q+1} \longrightarrow \mathbb{Q} \setminus \mathbb{Q}^-,
\]
\[
\sigma^{(q+1)} \longmapsto a(\sigma^{(q+1)}),
\]
where for each $\sigma^{(q+1)} \in S_{q+1}$,
\[
a(\sigma^{(q+1)}) =
\sum_{\sigma^{(q-1)} \in S_{q-1}}
(a_{q+1} \circ a_q)
\bigl(\sigma^{(q-1)}, \sigma^{(q+1)}\bigr)
\, a(\sigma^{(q-1)}).
\]

This completes the construction of the extension of the weight map to all simplices of the simplicial complex.

\subsection{Comparison of the total weights of 0-simplices and facets}

In the preceding sections, we extended the weight map to the simplicial complex, thereby assigning a concentration value to each simplex. The resulting weight distribution over the simplicial complex satisfies several structural properties, which we establish below in the form of propositions and theorem.

First, we compare the total weights of $q$-simplices and $(q-1)$-simplices.

\begin{proposition}\label{1prop:weight-compare}
Let $1 \le q \le dim(\Delta)$. The sum of the weights of all $q$-simplices equals the sum of the weights of all $(q-1)$-simplices that are not facets of $\Delta$. That is,
\begin{equation}\label{1eq:weight-compare}
\sum_{j=1}^{n_q} a\bigl(\sigma_j^{(q)}\bigr)
=
\sum_{i=1}^{n_{q-1}} \chi_i \, a\bigl(\sigma_i^{(q-1)}\bigr),
\end{equation}
where
\[
\chi_i =
\begin{cases}
1, & \text{if } \sigma_i^{(q-1)} \text{ is not a facet of } \Delta, \\[6pt]
0, & \text{if } \sigma_i^{(q-1)} \text{ is a facet of } \Delta,
\end{cases}
\quad \text{for } 1 \le i \le n_{q-1},
\]
and $n_q$ denotes the number of $q$-simplices of $\Delta$.
\end{proposition}

\begin{proof}
We begin by expanding the definition of the weight of each $q$-simplex in terms of the weights of its $(q-1)$-faces. By definition, this gives
\[
\sum_{j=1}^{n_q} a(\sigma_j^{(q)}) 
=
\sum_{j=1}^{n_q} \sum_{i=1}^{n_{q-1}} 
a_q(\sigma_i^{(q-1)}, \sigma_j^{(q)}) \, a(\sigma_i^{(q-1)}).
\]

Since all summands are non-negative, we may interchange the order of summation. This yields
\[
\sum_{j=1}^{n_q} a(\sigma_j^{(q)}) 
=
\sum_{i=1}^{n_{q-1}} 
\left(
\sum_{j=1}^{n_q} 
a_q(\sigma_i^{(q-1)}, \sigma_j^{(q)})
\right)
a(\sigma_i^{(q-1)}).
\]

We now examine the inner sum for a fixed $(q-1)$-simplex $\sigma_i^{(q-1)}$. If $\sigma_i^{(q-1)}$ is not a facet of $\Delta$, then it is contained in at least one $q$-simplex, and hence $k_{\sigma_i^{(q-1)}} \neq 0$. By the defining property of the map $a_q$, we have
\[
\sum_{j=1}^{n_q} 
a_q(\sigma_i^{(q-1)}, \sigma_j^{(q)}) = 1,
\]
and therefore
\[
\sum_{j=1}^{n_q} 
a_q(\sigma_i^{(q-1)}, \sigma_j^{(q)}) \, a(\sigma_i^{(q-1)})
=
a(\sigma_i^{(q-1)}).
\]

On the other hand, if $\sigma_i^{(q-1)}$ is a facet of $\Delta$, then it is not contained in any $q$-simplex, so $k_{\sigma_i^{(q-1)}} = 0$. In this case,
\[
\sum_{j=1}^{n_q} 
a_q(\sigma_i^{(q-1)}, \sigma_j^{(q)}) = 0,
\]
and consequently
\[
\sum_{j=1}^{n_q} 
a_q(\sigma_i^{(q-1)}, \sigma_j^{(q)}) \, a(\sigma_i^{(q-1)}) = 0.
\]

Combining both cases, we may write
\[
\sum_{j=1}^{n_q} a(\sigma_j^{(q)}) 
=
\sum_{i=1}^{n_{q-1}} \chi_i \, a(\sigma_i^{(q-1)}),
\]
where
\[
\chi_i =
\begin{cases}
1, & \text{if } k_{\sigma_i^{(q-1)}} \neq 0 \; (\text{that is, } \sigma_i^{(q-1)} \text{ is not a facet of } \Delta), \\[6pt]
0, & \text{if } k_{\sigma_i^{(q-1)}} = 0 \; (\text{that is, } \sigma_i^{(q-1)} \text{ is a facet of } \Delta).
\end{cases}
\]
This completes the proof.
\end{proof}

Next, we derive a corollary expressing the total weight of $q$-simplices in terms of the total weight of $(q-1)$-simplices and the weight of $(q-1)$-dimensional facets.

\begin{corollary}\label{1cor:weight-compare-facet}
The sum of the weights of all $q$-simplices together with the sum of the weights of the $(q-1)$-dimensional facets is equal to the sum of the weights of the $(q-1)$-simplices.

Let $1 \le q \le dim(\Delta)$. Then
\begin{equation}\label{1eq:weight-compare-facet}
\sum_{i=1}^{n_q} a(\sigma_i^{(q)}) 
+
\sum_{\substack{\sigma \text{ is a facet of } \Delta \\ dim(\sigma) = q-1}} a(\sigma)
=
\sum_{i=1}^{n_{q-1}} a(\sigma_i^{(q-1)}).
\end{equation}
\end{corollary}

\begin{proof}
We begin with equation~\eqref{1eq:weight-compare} from Proposition~\ref{1prop:weight-compare}, which states that
\[
\sum_{i=1}^{n_q} a(\sigma_i^{(q)}) 
= 
\sum_{i=1}^{n_{q-1}} \chi_i \, a(\sigma_i^{(q-1)}),
\]
where $\chi_i = 1$ if $\sigma_i^{(q-1)}$ is not a facet of $\Delta$, and $\chi_i = 0$ if $\sigma_i^{(q-1)}$ is a facet.

We now decompose the sum on the right-hand side into contributions from all $(q-1)$-simplices and those which are facets. Writing this explicitly, we obtain
\[
\sum_{i=1}^{n_{q-1}} \chi_i \, a(\sigma_i^{(q-1)}) 
=
\sum_{i=1}^{n_{q-1}} a(\sigma_i^{(q-1)}) 
-
\sum_{\substack{\sigma \text{ is a facet of } \Delta \\ dim(\sigma)=q-1}} a(\sigma).
\]

Substituting this expression into the left-hand side and rearranging the terms gives
\begin{equation}\label{1eq:weight-compare-facet2}
\sum_{i=1}^{n_q} a(\sigma_i^{(q)}) 
+
\sum_{\substack{\sigma \text{ is a facet of } \Delta \\ dim(\sigma)=q-1}} a(\sigma)
=
\sum_{i=1}^{n_{q-1}} a(\sigma_i^{(q-1)}).
\end{equation}
\end{proof}

The following corollary compares the total weight of $q$-simplices with the total weight of $0$-simplices.

\begin{corollary}\label{1cor:weight-compare-0-q}
Let $1 \le q \le dim(\Delta)$. The total weight of the $q$-simplices, together with the weights of all facets of dimension less than $q$, is equal to the total weight of the $0$-simplices. That is,
\[
\sum_{i=1}^{n_q} a\bigl(\sigma_i^{(q)}\bigr)
+
\sum_{\substack{\sigma \text{ is a facet of } \Delta \\ dim(\sigma) < q}} a(\sigma)
=
\sum_{i=1}^{n_0} a(\sigma_i^{(0)}).
\]
\end{corollary}

\begin{proof}
We apply Corollary~\ref{1cor:weight-compare-facet} successively for decreasing dimensions, starting from dimension $q$ and proceeding down to dimension $1$.

For $1\le q \le dim(\Delta)$, equation~\eqref{1eq:weight-compare-facet} gives
\[
\sum_{i=1}^{n_q} a(\sigma_i^{(q)}) 
+
\sum_{\substack{\sigma \text{ is a facet of } \Delta \\ dim(\sigma)=q-1}} a(\sigma)
=
\sum_{i=1}^{n_{q-1}} a(\sigma_i^{(q-1)}).
\]

Applying the equation \ref{1eq:weight-compare-facet2} to the right-hand side for dimension $q-1$, we obtain
\[
\sum_{i=1}^{n_{q-1}} a(\sigma_i^{(q-1)}) 
=
\sum_{i=1}^{n_{q-2}} a(\sigma_i^{(q-2)}) 
-
\sum_{\substack{\sigma \text{ is a facet of } \Delta \\ dim(\sigma)=q-2}} a(\sigma).
\]

Substituting this expression into the previous equation and proceeding iteratively over simplices of dimensions $q-1, q-2, \dots, 1$, we eventually obtain an expression involving only $0$-simplices. After the final step, we are left with

\[
\sum_{i=1}^{n_q} a(\sigma_i^{(q)})
+
\sum_{\substack{\sigma \text{ is a facet of } \Delta \\ dim(\sigma) < q}} a(\sigma)
=
\sum_{i=1}^{n_0} a(\sigma_i^{(0)}),
\]
which establishes the result.
\end{proof}

Finally, we prove a theorem establishing the equality of total weight between the facets and the $0$-simplices.

\begin{theorem}\label{1thm:facet-vertex-weight}
The total weight of all facets of $\Delta$ is equal to the total weight of the $0$-simplices. That is,
\[
\sum_{\substack{\sigma \text{ is a facet of } \Delta}} a(\sigma)
=
\sum_{i=1}^{n_0} a(\sigma_i^{(0)}).
\]
\end{theorem}

\begin{proof}
Let $q = dim(\Delta)$ in Corollary~\ref{1cor:weight-compare-0-q}. Since there are no simplices of dimension greater than $q$, every facet of $\Delta$ has dimension less than or equal to $q$.

Therefore, the sum
\[
\sum_{\substack{\sigma \text{ is a facet of } \Delta \\ dim(\sigma) \le q}} a(\sigma)
\]
is exactly the sum of the weights of all facets of $\Delta$.

Substituting this into the statement of Corollary~\ref{1cor:weight-compare-0-q}, we obtain
\[
\sum_{\substack{\sigma \text{ is a facet of } \Delta}} a(\sigma)
=
\sum_{i=1}^{n_0} a(\sigma_i^{(0)}),
\]
which completes the proof.
\end{proof}

The weight distribution over a simplicial complex implies that the total weight of all facets is equal to the total weight of the $0$-simplices. This reflects the interpretation that weight is initially assigned to vertices and is subsequently distributed across facets as higher-dimensional simplices are formed.

A CM in which every simplex carries a concentration value is called a weighted CM. In what follows, CM will always denote a weighted CM, unless explicitly stated otherwise.

\subsection{Inference of higher-order simplices}

Within the CM framework, we introduce a boundary-based deterministic rule for admitting higher-order simplices into the complex. We assume that the realization of a higher-order simplex is supported by sufficiently large boundary weights.
For $q \ge 2$, the inclusion of a combinatorially feasible $q$-simplex $\sigma^{(q)}$ is determined solely from the weights of its boundary $(q-1)$-simplices.

Let
\[
E_q = \{ \sigma^{(q)} \in \mathcal{P}(V) 
\mid \mathcal{B}(\sigma^{(q)}) \subseteq S_{q-1} \}
\]
denote the set of all combinatorially feasible $q$-simplices whose boundary faces are already present in $S_{q-1}$.

\begin{definition}[Inclusion rule]
Let $2 \leq q \leq dim(\Delta)$. An inclusion rule at dimension $q$ is a map
\[
\Phi_q : E_q \to \{0,1\}
\]
such that the family of $q$-simplices admitted in $\Delta$ is defined by
\[
S_q := \{ \sigma^{(q)} \in E_q \mid \Phi_q(\sigma^{(q)}) = 1 \}.
\]
\end{definition}

For $\sigma^{(q)} \in E_q$, define the aggregated boundary weight
\[
W_q(\sigma^{(q)}) 
= \sum_{\tau \in \mathcal{B}(\sigma^{(q)})} a(\tau).
\]

At dimension $q$, define the reference weight level of $(q-1)$-simplices by
\[
\bar{a}_{q-1} 
:= \frac{1}{|S_{q-1}|} 
\sum_{\tau \in S_{q-1}} a(\tau).
\]

The dimension-dependent threshold is then defined as
\[
\theta_q := (q+1)\,\bar{a}_{q-1}
= \frac{q+1}{|S_{q-1}|}
\sum_{\tau \in S_{q-1}} a(\tau).
\]

We define a specific weight-based inclusion rule by
\[
\Phi_q(\sigma^{(q)}) = 1 
\iff
W_q(\sigma^{(q)}) > \theta_q.
\]

The family of admitted $q$-simplices is therefore
\[
S_q 
= \{ \sigma^{(q)} \in E_q 
\mid W_q(\sigma^{(q)}) > \theta_q \}.
\]

This defines a weight-based CM model in which higher-order simplices are introduced through a deterministic threshold criterion.

In the clique complex construction, the existence of a $q$-simplex is determined purely combinatorially: a simplex is present if and only if the vertices of the simplex induce a complete subgraph of the $1$-skeleton of $\Delta$.

In the CM, this combinatorial condition is retained but refined using weight information. While the presence of all boundary $(q-1)$-simplices ensures that the simplex is combinatorially admissible, an additional structural constraint is imposed: a $q$-simplex is included only if its aggregated boundary weight exceeds the prescribed threshold $\theta_q$.

The threshold $\theta_q$ can be interpreted in terms of the average weight of $(q-1)$-simplices. The quantity $\bar{a}_{q-1}$ denotes the average weight of all $(q-1)$-simplices currently present in the simplicial complex. Since a $q$-simplex has exactly $(q+1)$ boundary $(q-1)$-faces, the value $(q+1)\bar{a}_{q-1}$ represents the total weight obtained when each boundary face carries the average weight. The threshold therefore serves as a reference level for evaluating the collective boundary support of a candidate simplex. A $q$-simplex is included only when its aggregated boundary weight exceeds this thresold.

Consequently, in the CM the deterministic threshold criterion filters out purely combinatorial closures and includes only those simplices whose boundary weight exceeds the prescribed threshold.

\section{Centrality indices in the combinatorial metaplex}\label{1sec:centralities in CM}
Having established the axiomatic framework of the CM, defined by a concentration map on the vertex set and its consistent extension to higher-dimensional simplices, we now turn to the quantitative characterisation of its combinatorial structure. The concentration map fundamentally modifies the induced adjacency structure on the simplicial complex, necessitating the definition of centrality indices adapted to the weighted combinatorial framework.

\subsection{Background: centrality in weighted graphs}

To motivate the definition of centrality in a CM, we recall the development of centrality indices in graph theory, from unweighted graphs to fully weighted graphs.
For the graph-based centrality definitions that follow, 
we reindex the vertex set as 
$V = \{v_1, v_2, \ldots, v_n\}$ 
so that the indexing of vertices aligns naturally with the 
entries of the adjacency matrix.

\begin{itemize}
\item[{(a)}]{ Unweighted graphs:}
Let $G = (V, E)$ be a simple graph and $V=\{v_1,v_2,\cdots,v_n\}$ (consider this vertex set only for graph centralities). Two distinct vertices $v_i, v_j \in V$ are said to be adjacent if and only if the unordered pair $\{v_i, v_j\}$ is an element of the edge set $E$:
\[
v_i \sim v_j \iff \{v_i, v_j\} \in E.
\]
The primary centrality index is the {degree}, defined as the cardinality of the neighbourhood of a vertex (that is, the set containing adjacent vertices). In matrix form, the degree of $v_i$ is given by the row-sum of the adjacency matrix \cite{estrada2015first}:

\begin{equation}\label{1eq:degre}
k(v_i) = \sum_{j=1}^n a_{ij},
\end{equation}
where $A = [a_{ij}]_{n \times n}$ is the binary adjacency matrix with $a_{ij} = 1$ if $\{v_i, v_j\} \in E$ and $a_{ij} = 0$ otherwise.

\item[(b)] {Edge-weighted graphs:}
Let $G = (V, E,w)$ be a graph equipped with an edge-weight function $w : E \to \mathbb{R}$. The degree centrality is generalised to the {vertex strength}, defined as the sum of the weights of edges incident to a vertex \cite{barrat2004architecture}:
\begin{equation}\label{1eq:weighted-degre}
s(v_i) = \sum_{j=1}^n w_{ij},
\end{equation}
where $W = [w_{ij}]_{n \times n}$ is the weighted adjacency matrix, with $w_{ij} = w(\{v_i, v_j\})$ if $a_{ij} = 1$ and $w_{ij} = 0$ otherwise. In this formulation, vertices are treated as unweighted elements.

\item[(c)] {Vertex-weighted graphs:}
To incorporate heterogeneity at the level of vertices, centrality definitions have been extended to the class of vertex-weighted graphs $G = (V, E, W)$, where a weight function $W : V \to \mathbb{R}$ assigns a non-negative real value to each vertex \cite{singh2020node}. In this framework, the centrality of a vertex $v_i$ depends on the weights of its neighbours:
\begin{equation}
C(v_i) = \sum_{j=1}^n W(v_j)\, a_{ij},
\end{equation}
where $A = [a_{ij}]_{n \times n}$ is the binary adjacency matrix.

\item [(d)] {Fully weighted graphs:}
The most general graph and direct precursor to the CM framework is the fully weighted graph $G=(V,E,W,w)$, in which both vertices and edges are weighted by functions $W : V \to \mathbb{R}$ and $w : E \to \mathbb{R}$. The degree centrality in this case couples vertex weights with edge weights:
\begin{equation}\label{1eq:fully-weig-degre}
C_f(v_i) = \sum_{j=1}^n W(v_j)\, w_{ij},
\end{equation}
where $A_w = [w_{ij}]_{n \times n}$ is the weighted adjacency matrix.

This formulation ensures consistency with the preceding special cases. Imposing edge homogeneity, that is, $w(e) = 1$ for all $e \in E$, recovers the degree centrality of vertex-weighted graphs. Imposing vertex homogeneity, that is $W(v) = 1$ for all $v \in V$, produces the formulation for edge-weighted graphs. Simultaneously imposing both conditions reduces the index to the classical degree centrality of unweighted graphs.
\end{itemize}

In the succeeding section, we extend the fully weighted graph centrality model to the CM by replacing scalar vertex weight with simplex weight and edge weight with facet-induced adjacency strength.

\subsection{Adjacency and degree in the CM}

In this section, we introduce a family of centrality maps defined on the set of $q$-simplices. Each centrality map assigns to a simplex a real value determined jointly by its incidence structure and the weight map, and satisfies structural properties compatible with the CM framework. These maps induce a ranking on simplices and provide a quantitative characterisation of their structural importance within the CM, thereby distinguishing the CM from purely topological simplicial complexes.

We focus on centrality indices derived from the incidence structure (degree) and from the shortest-path distance (closeness and harmonic). In the CM framework, these indices are defined with respect to the weighted adjacency relations induced by the concentration map on $\mathcal{K}_c$. We now formalise this construction.

\subsubsection{Facet-induced adjacency:}

Let $\mathcal{F}(\mathcal{K}_c)$ denote the set of facets of the CM,
\[
\mathcal{F}(\mathcal{K}_c) =
\left\{
\gamma \in \Delta \;\middle|\; \nexists \tau \in \Delta  \text{ such that } \gamma \subsetneq \tau
\right\},
\]
and let the set of $q$-simplices be
\[
S_q(\mathcal{K}_c) = \{ \sigma \in \Delta \mid \dim(\sigma) = q \}.
\]

Henceforth, $\sigma$ denotes a $q$-simplex unless otherwise stated. We formalise adjacency between simplices in terms of their joint containment in a facet. Two distinct $q$-simplices $\sigma_i, \sigma_j \in S_q(\mathcal{K}_c)$ are said to be adjacent if they are both contained in a common facet. Formally, we define the adjacency relation by
\[
\sigma_i \sim_F \sigma_j
\iff
\exists \gamma \in \mathcal{F}(\mathcal{K}_c)
\text{ such that }
\sigma_i \subsetneq \gamma
\text{ and }
\sigma_j \subsetneq \gamma.
\]
This definition implies that the adjacency between simplices of dimension $q$ is mediated by a facet containing both \cite{raj2025study}.

\begin{definition}
The {strength of adjacency} between two $q$-simplices $\sigma_i$ and $\sigma_j$ is defined by
\[
S(\sigma_i, \sigma_j)
=
\max \left\{
a(\gamma)
\;\middle|\;
\gamma \in \mathcal{F}(\mathcal{K}_c),\;
\sigma_i \cup \sigma_j \subseteq \gamma
\right\}.
\]
We set $S(\sigma_i, \sigma_j) = 0$ if $\sigma_i$ and $\sigma_j$ are not adjacent or if $i = j$.
\end{definition}

The strength of adjacency assigns to each adjacent pair the weight of a mediating facet of maximal concentration, thereby providing a quantitative measure of facet-supported coupling between simplices.

We represent adjacency at dimension $q$ by matrices of size $n_q \times n_q$. In a simplicial complex, adjacency matrices at level $q$ based on upper adjacency were introduced in \cite{estrada2018}, here we define facet-induced adjacency.

\begin{definition}
The {simplicial adjacency matrix} at level $q$ is denoted
\[
A^{(q)} = [a_{ij}^{(q)}],
\]
where
\[
a_{ij}^{(q)} =
\begin{cases}
1, & \text{if } i \neq j \text{ and } \sigma_i \sim_F \sigma_j, \\
0, & \text{otherwise}.
\end{cases}
\]
\end{definition}

When $q=0$ and we restrict to the $1$-skeleton of the simplicial complex, we recover the adjacency matrix defined for the unweighted graph. Inspired by the weighted simplicial adjacency matrix in \cite{raj2025study}, we define a facet-induced weighted adjacency matrix as follows.
\begin{definition}
The {weighted adjacency matrix} at level $q$ is denoted
\[
A_w^{(q)} = [w_{ij}^{(q)}],
\]
where
\[
w_{ij}^{(q)} =
\begin{cases}
S(\sigma_i, \sigma_j), & \text{if } i \neq j \text{ and } \sigma_i \sim_F \sigma_j, \\
0, & \text{otherwise}.
\end{cases}
\]
\end{definition}

When $q=0$ and we restrict to the $1$-skeleton of the simplicial complex, we recover the weighted adjacency matrix of a weighted graph. The matrix $A^{(q)}$ encodes the binary adjacency structure of the simplicial complex, while $A_w^{(q)}$ encodes the corresponding weighted adjacency determined by the weight map $a$. Together, these matrices distinguish the unweighted combinatorial structure from its weighted counterpart within the CM framework.

Before introducing degree centrality, we first formalise the definition of centrality at dimension $q$ within the combinatorial metaplex framework.

\begin{definition}
Let $\mathcal{K}_c$ be a CM, and let $S_q$ denote the set of all $q$-simplices of $\Delta$. A real-valued function $C$, defined as \(C : S_q \to \mathbb{R}\), is said to be a centrality index on $\mathcal{K}_c$ at level $q$ if it induces at least a semi-order on the elements of $S_q$.

\noindent That is, if $\sigma_i, \sigma_j, \sigma_k \in S_q$ and $\le$ is the total order relation on $\mathbb{R}$, then:
\begin{enumerate}
    \item[(i)] Reflexivity: $C(\sigma_i) \le C(\sigma_i),\ \forall \sigma_i \in S_q$.
    \item[(ii)] Transitivity: if $C(\sigma_i) \le C(\sigma_j)$ and $C(\sigma_j) \le C(\sigma_k)$, then $C(\sigma_i) \le C(\sigma_k)$.
    \item[(iii)] Anti-symmetry: if $C(\sigma_i) \le C(\sigma_j)$ and $C(\sigma_j) \le C(\sigma_i)$, then $C(\sigma_i) = C(\sigma_j)$.
    \item[(iv)] Comparability: either $C(\sigma_i) \le C(\sigma_j)$ or $C(\sigma_j) \le C(\sigma_i)$.
\end{enumerate}
\end{definition}

We define degree centrality based on facet-induced adjacency, incorporating both the number of adjacent simplices and the strength of their coupling.

\medskip

\noindent (a) Simplicial degree $(k(\sigma_i))$ \\
The simplicial degree of a $q$-simplex $\sigma_i$ is defined as the cardinality of the set of $q$-simplices adjacent to it. Analytically, it is given by the row sum of the simplicial adjacency matrix:
\[
k(\sigma_i) = \sum_{j=1}^{n_q} a_{ij}^{(q)}.
\]
This quantity depends only on the incidence structure of the simplicial complex and is independent of the weight map.

\medskip

\noindent (b) Weighted degree $(D(\sigma_i))$ \\
Following the formulation in \cite{raj2025study} and the fully weighted graph degree in \eqref{1eq:fully-weig-degre}, the weighted degree of a $q$-simplex $\sigma_i$ is defined by
\[
D(\sigma_i)
=
\sum_{j=1}^{n_q}
a(\sigma_j)\, w_{ij}^{(q)}.
\]
This quantity measures the aggregate contribution of adjacent simplices, weighted by both their simplex weights and the strengths of the facet-mediated adjacencies.

If $\dim(\Delta) = 1$ and $q = 0$, then the weighted degree definition reduces to the degree centrality of a vertex in a fully weighted graph, as given in \eqref{1eq:fully-weig-degre}.
Just as centrality in graphs can be expressed as a linear operator acting on a vertex-weight vector, the CM framework admits a higher-order generalisation in which concentration-induced adjacency operators act on weight functions defined over simplices.

\medskip

\noindent (c) Combined degree centrality $(D^{\alpha}(\sigma_i))$ \\
To balance the combinatorial contribution of adjacency cardinality with the contribution of adjacency strength, we introduce a one-parameter family of degree centralities.

Let $\alpha \in [0,1]$. The combined degree centrality of a $q$-simplex $\sigma_i \in S_q(\mathcal{K}_c)$ is defined by
\[
D^{\alpha}(\sigma_i)
=
\bigl(k(\sigma_i)\bigr)^{1-\alpha}
\cdot
\bigl(D(\sigma_i)\bigr)^{\alpha}.
\]

The parameter $\alpha$ regulates the relative contribution of adjacency cardinality and facet-induced coupling strength:
\begin{itemize}
    \item[(i)] {Topological limit}: ($\alpha = 0$): $D^{\alpha}(\sigma_i) = k(\sigma_i)$.
    \item[(ii)] {Weighted limit}: ($\alpha = 1$): $D^{\alpha}(\sigma_i) = D(\sigma_i)$.
    \item[(iii)] {Intermediate regime}: ($0 < \alpha < 1$): $D^{\alpha}(\sigma_i)$ interpolates between the two.
\end{itemize}

This formulation defines a continuous family of centrality indices on the CM, allowing simplices to be compared with respect to both their combinatorial adjacency and the concentration structure of the facets that support those adjacencies \cite{wei2012degree, yustiawan2015degree}.

\subsection{Weighted walks and distances}

In the CM, connectivity between simplices is determined not only by facet-induced adjacency, but also by sequences of facets that mediate successive adjacencies. Unlike in graphs, where an edge uniquely determines a transition between two vertices, a pair of simplices in a simplicial complex may be jointly contained in multiple facets. Consequently, a walk between simplices is specified by both the sequence of simplices and the choice of mediating facets at each step. We define the weighted walk following the definition in \cite{raj2025study}.

\begin{definition}
Let $\sigma_{\mathrm{start}}, \sigma_{\mathrm{end}} \in S_q(\mathcal{K}_c)$. A {weighted walk} $W$ from $\sigma_{\mathrm{start}}$ to $\sigma_{\mathrm{end}}$ is an alternating sequence of $q$-simplices and facets,
\[
W:\quad \sigma_{i_1},\; \gamma_1,\; \sigma_{i_2},\; \gamma_2,\; \dots,\; \gamma_n,\; \sigma_{i_{n+1}},
\]
such that:
\begin{enumerate}
    \item[(i)] $\sigma_{i_1} = \sigma_{\mathrm{start}}$ and $\sigma_{i_{n+1}} = \sigma_{\mathrm{end}}$.
    \item[(ii)] For each $1 \leq l \leq n$, the facet $\gamma_l \in \mathcal{F}(\mathcal{K}_c)$ satisfies
    \[
    \sigma_{i_l} \subsetneq \gamma_l
    \quad \text{and} \quad
    \sigma_{i_{l+1}} \subsetneq \gamma_l.
    \]
    \item[(iii)] For each $1 \leq l \leq n$, the mediating facet $\gamma_l$ satisfies
    \[
    a(\gamma_l) = S(\sigma_{i_l}, \sigma_{i_{l+1}}).
    \]
\end{enumerate}
\end{definition}

This definition ensures that each consecutive pair of simplices in the walk is connected by a facet of maximal weight among all admissible mediators. Two simplices are said to be {connected} if there exists a weighted walk between them.

Let $W$ be a weighted walk as defined above. The {length} of $W$ is given by
\[
\mathcal{L}(W)
=
\sum_{l=1}^{n}
\frac{1}{a(\gamma_l)\, a(\sigma_{i_{l+1}})}.
\]

This definition reflects the modelling assumption that higher weights of mediating facets and destination simplices indicate stronger combinatorial influence in the formation of higher-order structure. Accordingly, the contribution of each step in the walk is taken to be inversely proportional to the product $a(\gamma_l)\, a(\sigma_{i_{l+1}})$. Under this choice, steps involving larger weights contribute smaller increments to the total length.

Consequently, walks that pass through highly concentrated facets and simplices have smaller total length, ensuring that shortest paths preferentially follow sequences supported by higher concentration. This definition is therefore consistent with the principle that greater boundary concentration facilitates higher-order interaction.

\begin{definition}
The shortest path distance between two simplices $\sigma_i, \sigma_j \in S_q(\mathcal{K}_c)$ is defined by
\[
d(\sigma_i, \sigma_j)
=
\min_{W} \mathcal{L}(W),
\]
where the minimum is taken over all weighted walks $W$ connecting $\sigma_i$ and $\sigma_j$.
\end{definition}

Because walk length is defined as a sum of strictly positive step contributions, the induced shortest-path function satisfies non-negativity and the triangle inequality.
Moreover, $d(\sigma_i,\sigma_j)=0$ if and only if $\sigma_i=\sigma_j$. 
However, because the step length depends on the destination simplex weight, the resulting distance is not symmetric. 


\begin{definition}
Let $\alpha \in [0,1]$. The {combined shortest path distance} between $\sigma_i, \sigma_j \in S_q(\mathcal{K}_c)$ is defined by
\[
d^{\alpha}(\sigma_i, \sigma_j)
=
\min_{W}
\left\{
\sum_{l=1}^{n}
\left(
\frac{1}{w^{(q)}_{i_l i_{l+1}} \, a(\sigma_{i_{l+1}})}
\right)^{\alpha}
\right\},
\]
where
\[
w^{(q)}_{i_l i_{l+1}} = S(\sigma_{i_l}, \sigma_{i_{l+1}}) = a(\gamma_l).
\]
\end{definition}

The parameter $\alpha$ regulates the relative contribution of facet weights and simplex weights in the computation of path length:
\begin{enumerate}
    \item[(i)] Weighted limit ($\alpha = 1$): $d^{\alpha}$ coincides with the distance induced by the length $\mathcal{L}(W)$.
    \item[(ii)] Intermediate regime ($0 < \alpha < 1$): $d^{\alpha}$ interpolates between purely combinatorial and weight-dominated distances.
    \item[(iii)] Topological limit ($\alpha = 0$): $d^{\alpha}$ reduces to the minimum number of steps in a facet-mediated walk.
\end{enumerate}

\subsection{Weighted closeness centrality}

We consider the adjacency structure on $S_q(\mathcal{K}_c)$ induced by facet-mediated weighted walks.

\begin{definition}
A connected component of $S_q(\mathcal{K}_c)$ is a maximal subset 
$\mathcal{D} \subseteq S_q(\mathcal{K}_c)$ such that for every 
$\sigma_i, \sigma_j \in \mathcal{D}$ there exists a finite weighted walk joining them.
\end{definition}

Closeness centrality quantifies how close a simplex is, on average, to all other simplices within its connected component, with respect to the shortest-path distance induced by weighted walks.

Let $\sigma_i \in S_q(\mathcal{K}_c)$ and denote by $\mathcal{D}(\sigma_i)$ the connected component containing $\sigma_i$. The {farness} of $\sigma_i$ is defined as
\[
F(\sigma_i)
=
\sum_{\substack{\sigma_j \in \mathcal{D}(\sigma_i) \\ \sigma_j \neq \sigma_i}}
d(\sigma_i,\sigma_j).
\]

\begin{definition}
Let $\sigma_i \in S_q(\mathcal{K}_c)$. 
The weighted closeness centrality (CC) of $\sigma_i$ is defined by
\[
CC(\sigma_i)
=
\frac{1}{\displaystyle 
\sum_{\substack{\sigma_j \in \mathcal{D}(\sigma_i) \\ \sigma_j \neq \sigma_i}}
d(\sigma_i,\sigma_j)}.
\]
\end{definition}

Thus, simplices with smaller aggregate distance to the remaining simplices in their component receive larger centrality values.

In disconnected settings, the reciprocal formulation above becomes degenerate when distances outside the component are infinite. A standard alternative is the harmonic formulation, which naturally accommodates disconnected structures.

\begin{definition}
Let $\sigma_i \in S_q(\mathcal{K}_c)$. 
The weighted harmonic centrality (HC) of $\sigma_i$ is defined by
\[
HC(\sigma_i)
=
\sum_{\substack{\sigma_j \in S_q(\mathcal{K}_c) \\ \sigma_j \neq \sigma_i}}
\frac{1}{d(\sigma_i,\sigma_j)},
\]
where terms with $d(\sigma_i,\sigma_j)=\infty$ are taken to contribute $0$.
\end{definition}

We now introduce a parameter-dependent version based on the combined weighted distance $d^{\alpha}$.

\begin{definition}
Let $\alpha \in [0,1]$ and $\sigma_i \in S_q(\mathcal{K}_c)$. 
The combined weighted closeness centrality of $\sigma_i$ is defined by
\[
CC^{\alpha}(\sigma_i)
=
\frac{1}{\displaystyle 
\sum_{\substack{\sigma_j \in \mathcal{D}(\sigma_i) \\ \sigma_j \neq \sigma_i}}
d^{\alpha}(\sigma_i,\sigma_j)}.
\]
\end{definition}

Analogously, the combined harmonic centrality is defined by
\[
HC^{\alpha}(\sigma_i)
=
\sum_{\substack{\sigma_j \in S_q(\mathcal{K}_c) \\ \sigma_j \neq \sigma_i}}
\frac{1}{d^{\alpha}(\sigma_i,\sigma_j)},
\]
with infinite distances contributing $0$.

The parameter $\alpha$ regulates the relative influence of step count and facet-mediated transition strength in the distance evaluation. 
When $\alpha = 0$, all transitions contribute equally and the metric reduces to the minimal number of facet-mediated steps. 
When $\alpha = 1$, distances are determined entirely by the weights assigned to simplices and their mediating facets along the walk.
For $0 < \alpha < 1$, the resulting centrality interpolates between the step-based regime and the fully weight-driven regime.

We note that betweenness centrality is not defined in the present framework. The classical definition of betweenness relies on symmetric shortest-path distances. In this work, the length of a weighted walk is asymmetric, and consequently the induced distance function does not satisfy symmetry. A consistent extension of classical betweenness centrality is therefore not immediate, and the development of an appropriate asymmetric analogue remains a direction for future investigation.

\section{Illustration and results}\label{1sec:Illustration of CM}
In this section, we illustrate the CM framework developed in Section~\ref{1sec:CM framework} using a real-world network. Specifically, the CM is constructed from pairwise data obtained from the food web dataset available in the cheddar R package \cite{hudson2013cheddar}. We consider the food web of Tuesday Lake, Michigan, United States of America, whose pairwise network consists of 50 vertices and 263 edges.

In this network, the discrete structure assigned to each vertex represents the collection of all individuals belonging to the corresponding species, while the weight denotes the total biomass of that species, measured in units of $mg\,m^{-3}$. Based on these vertex weights, we define the fractional contribution map, under which the weight of a $(q-1)$-simplex is distributed equally among all its incident $q$-simplices for $1 \leq q \leq 3$. Using the threshold criterion introduced earlier, we identify the HoIs included in the CM, which we refer to as true HoIs in this work.

The clique complex of this network contains 50 $0$-simplices, 263 $1$-simplices, 170 $2$-simplices, 113 $3$-simplices, 49 $4$-simplices, 13 $5$-simplices and 2 $6$-simplices. Table~1 presents a comparison between the number of HoIs obtained from the clique complex and true HoIs obtained from the CM.

\begin{table}[h!]
\centering
\setlength{\tabcolsep}{6pt}
\renewcommand{\arraystretch}{1.20}
\begin{tabular}{|p{2.5cm}|p{3.5cm}|p{3cm}|}
\hline
\textbf{Dimension of a simplex} & \textbf{No. of simplices in the clique complex} & \textbf{No. of simplices in the CM} \\
\hline
0-simplices & 50  & 50  \\
\hline
1-simplices & 263 & 263 \\
\hline
2-simplices & 170 & 89  \\
\hline
3-simplices & 113 & 5   \\
\hline
4-simplices & 49  & 0   \\
\hline
5-simplices & 13  & 0   \\
\hline
6-simplices & 2   & 0   \\
\hline
\textbf{Total} & \textbf{610} & \textbf{407} \\
\hline
\end{tabular}
\caption{\footnotesize Comparison of the number of simplices across dimensions in the clique complex and the CM, highlighting that the CM identifies true HoIs.}
\end{table}

\begin{figure}[h!]
	\centering
	
	\begin{subfigure}{0.48\textwidth}
		\centering
		\includegraphics[width=\linewidth]{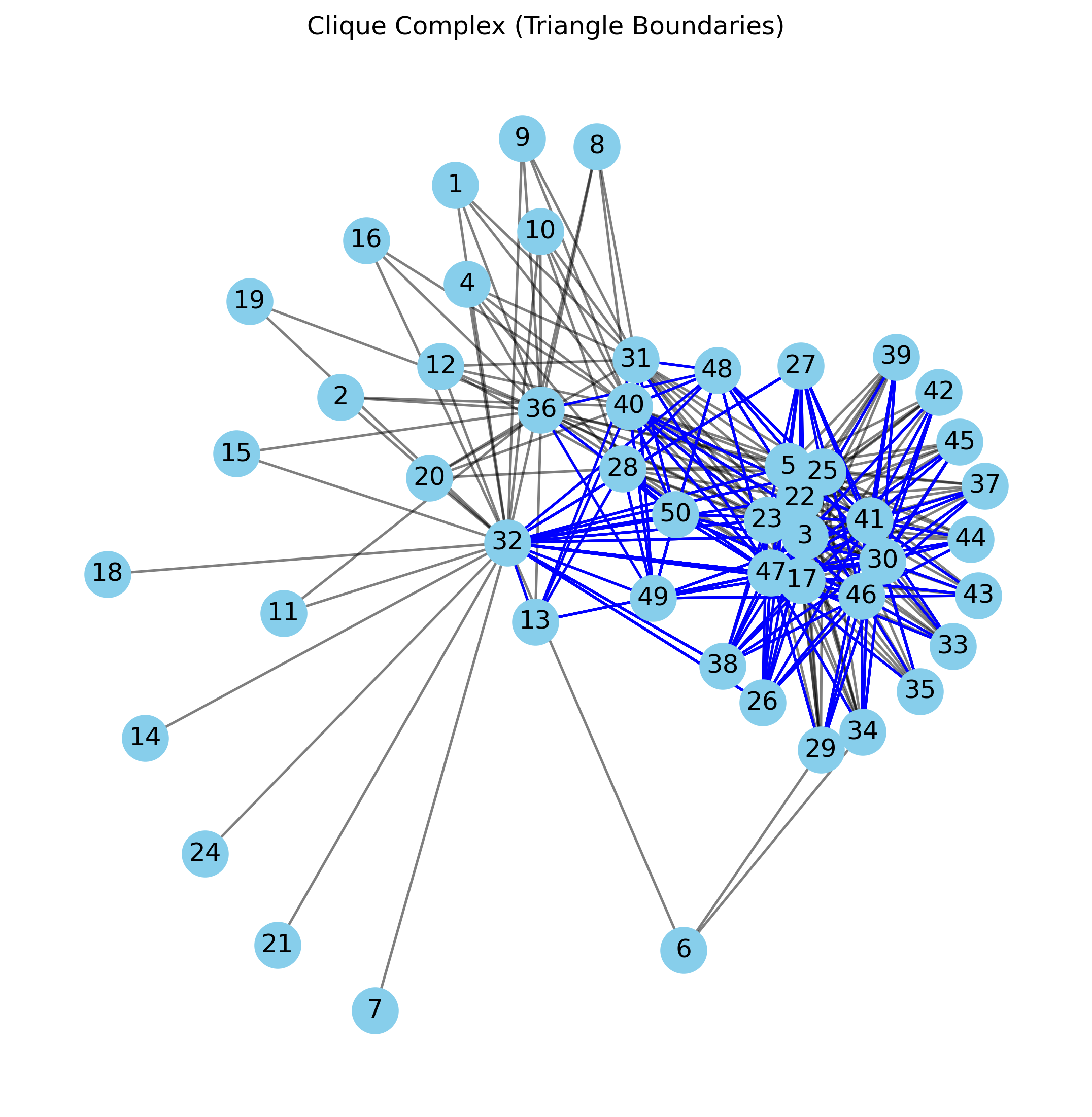}
		\caption{Clique complex}
	\end{subfigure}
	\hfill
	\begin{subfigure}{0.48\textwidth}
		\centering
		\includegraphics[width=\linewidth]{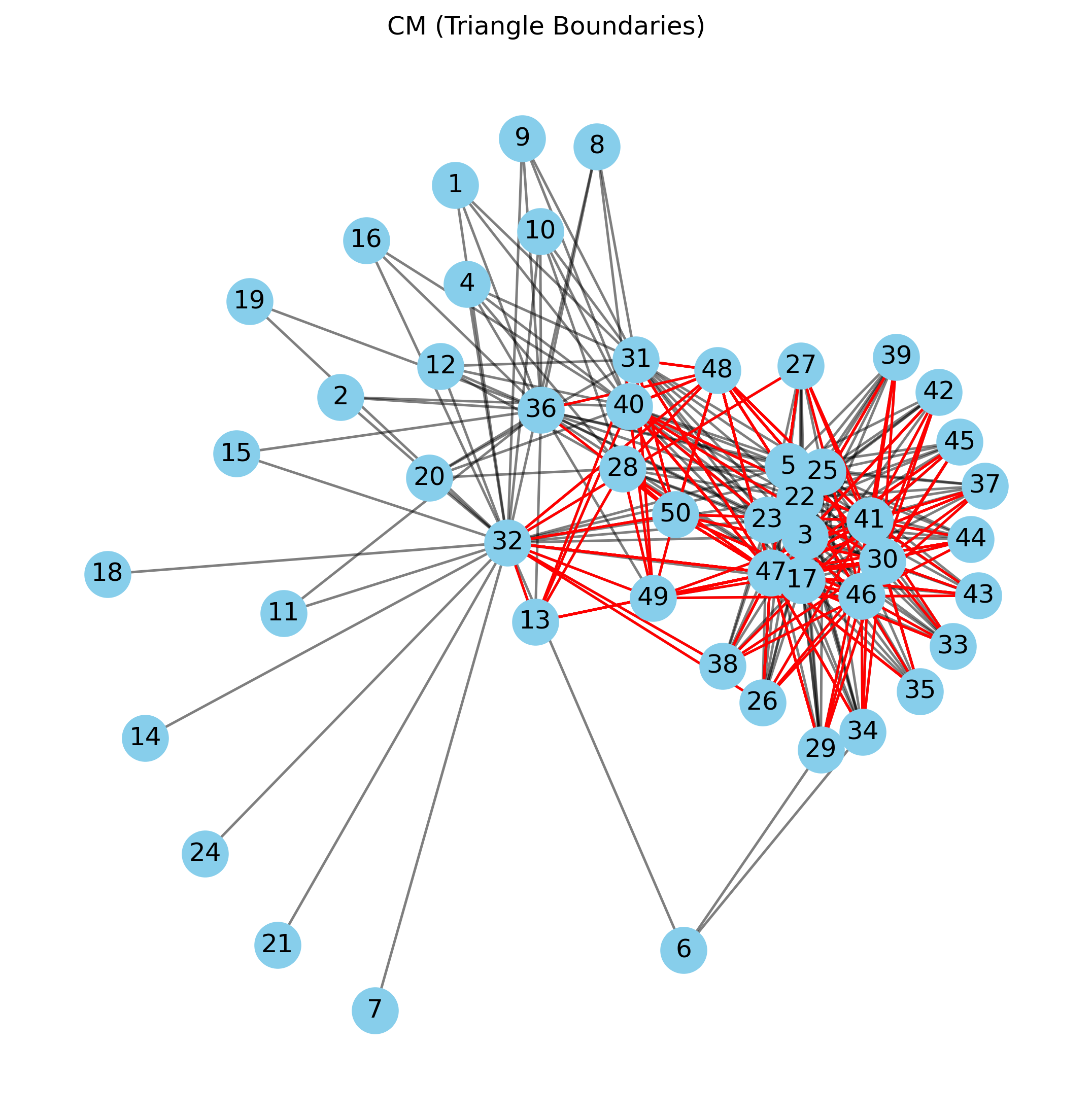}
		\caption{CM}
	\end{subfigure}
	
	\caption{\footnotesize Comparison of the 2-skeletons of the clique complex and the CM. The boundaries of triangles in the clique complex are shown in blue, while the boundaries of true triangles in the CM are shown in red.}
\end{figure}

\begin{table}[h!]
\centering
\setlength{\tabcolsep}{6pt}
\renewcommand{\arraystretch}{1.2}

\begin{tabular}{|p{2.5cm}|p{2cm}|p{2.7cm}|p{2.4cm}|p{2.8cm}|}
\hline
\textbf{Dimension  of a simplex} 
& \textbf{No.\ of facets} 
& \textbf{Sum of facet weights} 
& \textbf{No.\ of non-facets} 
& \textbf{Sum of non-facet weights} \\
\hline
0-simplices & 0   & 0            & 50  & 1171.372469 \\
\hline
1-simplices & 162 & 412.376647   & 101 & 758.9958 \\
\hline
2-simplices & 72  & 576.776620   & 17  & 158.8958 \\
\hline
3-simplices & 5   & 182.219202   & 0   & 0 \\
\hline
\end{tabular}
\caption{\footnotesize Summary of facet and non-facet simplices counts and their total weights across dimensions, verifying Theorem 3.5 by showing that the sum of vertex weights equals the sum of facet weights.}
\end{table}

\begin{table}[h!]
\centering
\setlength{\tabcolsep}{6pt}
\renewcommand{\arraystretch}{1.10}

\begin{tabular}{|c|p{1.6cm}|p{2.2cm}|p{2.3cm}|p{2.2cm}|}
\hline
\textbf{Edge} 
& \textbf{Edge weight} 
& \textbf{Simplicial degree: $k$ ($\alpha = 0$)} 
& \textbf{Combined degree: $D^\alpha$ ($\alpha = 0.5$)}
& \textbf{Weighted degree: $D^1$ ($\alpha = 1$)} \\
\hline

(47,48) & 34.8970 & 21 & 466.2371 & 10351.2867 \\
\hline
(48,50) & 21.4450 & 21 & 418.8909 & 8355.6944 \\
\hline
(47,50) & 16.5480 & 21 & 403.2099 & 7741.8187 \\
\hline
(41,47) & 17.6106 & 40 & 506.3649 & 6410.1363 \\
\hline
(28,47) & 16.6017 & 9  & 216.0144 & 5184.6920 \\
\hline
(46,47) & 16.6851 & 39 & 441.1313 & 4989.6619 \\
\hline
(28,50) & 3.1497  & 5  & 156.9601 & 4927.2966 \\
\hline
(31,47) & 15.1792 & 9  & 203.2567 & 4590.3667 \\
\hline
(32,47) & 15.3263 & 15 & 261.3505 & 4553.6043 \\
\hline
(31,50) & 1.7272  & 5  & 148.6004 & 4416.4158 \\
\hline

\multicolumn{5}{|c|}{$\vdots$} \\
\hline

(40,49) & 1.6479 & 2 & 19.2811 & 185.8814 \\
\hline
(27,30) & 1.0547 & 2 & 18.8169 & 177.0369 \\
\hline
(30,49) & 2.3517 & 2 & 18.4353 & 169.9292 \\
\hline
(30,38) & 0.8344 & 2 & 18.2916 & 167.2912 \\
\hline
(26,30) & 0.8132 & 2 & 18.2405 & 166.3572 \\
\hline

\end{tabular}
\caption{\footnotesize Top 10 and bottom 5 edges ranked by weighted degree $D^1$ ($\alpha = 1$). The table presents edge weights together with the simplicial degree $k$ ($\alpha = 0$) and the combined degree $D^\alpha$ ($\alpha = 0.5$).}
\end{table}

\begin{table}[h!]
	\centering
	\setlength{\tabcolsep}{6pt}
	\renewcommand{\arraystretch}{1.10}
	\begin{tabular}{|c|c|c|c|c|}
		\hline
		\textbf{Edge} & $\textbf{Edge weight}$ & $\mathbf{CC^{0}}$ & $\mathbf{CC^{0.5}}$ & $\mathbf{CC^{1}}$ \\
		\hline
		
		$(41,47)$ & 17.6106 & 0.006061 & 0.042212 & 0.160132 \\
		\hline
		$(47,48)$ & 34.8970 & 0.005587 & 0.041699 & 0.159768 \\
		\hline
		$(46,47)$ & 16.6851 & 0.006024 & 0.041559 & 0.159610 \\
		\hline
		$(41,48)$ & 22.5076 & 0.004717 & 0.039746 & 0.158834 \\
		\hline
		$(48,50)$ & 21.4450 & 0.004292 & 0.039209 & 0.158776 \\
		\hline
		$(41,46)$ & 4.2957  & 0.004405 & 0.039274 & 0.158700 \\
		\hline
		$(46,48)$ & 21.5821 & 0.004630 & 0.039435 & 0.158689 \\
		\hline
		$(47,50)$ & 16.5480 & 0.005587 & 0.039007 & 0.158659 \\
		\hline
		$(28,47)$ & 16.6017 & 0.004000 & 0.038666 & 0.158603 \\
		\hline
		$(31,50)$ & 1.7272  & 0.003802 & 0.038249 & 0.158565 \\
		\hline
		
		\multicolumn{5}{|c|}{$\vdots$} \\
		\hline
		
		$(40,49)$ & 1.6479 & 0.003636 & 0.028457 & 0.137050 \\
		\hline
		$(27,30)$ & 1.0547 & 0.003846 & 0.028028 & 0.136933 \\
		\hline
		$(26,30)$ & 0.8132 & 0.003846 & 0.027803 & 0.136921 \\
		\hline
		$(30,38)$ & 0.8344 & 0.003846 & 0.027822 & 0.136899 \\
		\hline
		$(30,49)$ & 2.3517 & 0.003968 & 0.027978 & 0.134071 \\
		\hline
		
	\end{tabular}

	\caption{Top 10 and bottom 5 edges ranked according to closeness centrality ($CC^{1}$). The corresponding values for $\alpha = 0, 0.5, 1$, together with the edge weights, are also reported.}
	
\end{table}

\begin{table}[h!]
\centering
\setlength{\tabcolsep}{6pt}
\renewcommand{\arraystretch}{1.15}

\begin{tabular}{|c|c|c|c|c|}
\hline
\textbf{Edge} & $\textbf{Edge weight}$ & $\mathbf{HC^{0}}$ & $\mathbf{HC^{0.5}}$ & $\mathbf{HC^{1}}$ \\
\hline

$(47,48)$ & 34.8970 & 60.5000 & 795.0753 & 13893.3599 \\
\hline
$(47,50)$ & 16.5480 & 60.5000 & 719.8201 & 13658.1812 \\
\hline
$(48,50)$ & 21.4450 & 51.5000 & 728.4055 & 13610.8367 \\
\hline
$(28,50)$ & 3.1497  & 41.1667 & 680.7542 & 13354.9942 \\
\hline
$(28,47)$ & 16.6017 & 44.6667 & 688.3476 & 13024.0892 \\
\hline
$(31,50)$ & 1.7272  & 41.1667 & 669.2698 & 12828.6989 \\
\hline
$(28,48)$ & 21.4987 & 41.1667 & 661.6855 & 12540.4101 \\
\hline
$(31,47)$ & 15.1792 & 44.6667 & 673.5648 & 12501.1194 \\
\hline
$(32,50)$ & 1.8743  & 42.1667 & 662.8607 & 12417.5455 \\
\hline
$(32,47)$ & 15.3263 & 50.1667 & 675.0535 & 12130.6257 \\
\hline

\multicolumn{5}{|c|}{$\vdots$} \\
\hline

$(40,49)$ & 1.6479 & 38.3333 & 354.9303 & 3278.5304 \\
\hline
$(27,30)$ & 1.0547 & 41.0000 & 336.8375 & 3044.8693 \\
\hline
$(30,38)$ & 0.8344 & 41.0000 & 332.8572 & 2966.6504 \\
\hline
$(26,30)$ & 0.8132 & 41.0000 & 332.4688 & 2958.9547 \\
\hline
$(30,49)$ & 2.3517 & 42.1667 & 335.1990 & 2921.9739 \\
\hline

\end{tabular}

\caption{Top 10 and bottom 5 edges ranked according to harmonic centrality ($HC^{1}$). The corresponding values for $\alpha = 0, 0.5, 1$, together with the edge weights, are also reported.}
\end{table}

Figure~1 illustrates the 2-skeletons of the clique complex and the CM for comparison: blue lines indicate the boundaries of triangles present in the clique complex, whereas red lines indicate the boundaries of true triangles present in the CM.

Table 2 verifies Theorem \ref{1thm:facet-vertex-weight} by reporting the number and total weights of facets and non-facet simplices across dimensions. The results confirm that the total weight of the vertices equals the total weight of the facets.

The analysis of the data using degree distributions of edges is shown in Table~3. It shows the top 10 edges having the highest weighted degree, along with their weights, simplicial degree, and combined degree ($\alpha = 0.5$).

Table~4 presents the top ten and bottom five edges with respect to closeness centrality under the weighted walk for $\alpha = 0$, $0.5$, and $1$. The edges are ranked according to the values obtained for $\alpha = 1$, and their corresponding weights are also reported. The shortest-path distances are computed using Dijkstra's algorithm.

Table~5 presents the top ten and bottom five edges with respect to harmonic centrality across the same range of $\alpha$ values. The edges are ranked according to the values obtained for $\alpha = 1$, and their corresponding weights are also reported.

We observe from Tables~3, 4, and 5 that among the top 10 edges, four edges are common across all three centrality indices, while the remaining edges vary. However, the last five edges are the same in all three centralities, with only a change in their order.

The following observations could be made from the above discussions:

\begin{enumerate}
    \item[(i)] A comparison between the simplices arising from the clique complex and those constructed under the CM shows that the CM yields a more precise and structurally meaningful representation for the inference of higher-order interactions (HoIs).

    \item[(ii)] The weight distribution across dimensions verifies the conservation principle established in the theorem, namely that the total weight assigned to vertices coincides exactly with the total weight carried by the facets.

    \item[(iii)] The centrality measures computed for different values of $\alpha$ illustrate how the relative importance of simplices shifts as the model transitions from a purely combinatorial regime to a weight-dominated regime.
\end{enumerate}
\section{Conclusion}

In this work, we introduce the combinatorial metaplex as a mathematical framework for modelling higher-order networks, in which interaction structure and intrinsic unit properties are coupled. In contrast to classical clique complex-based models, where HoIs included solely from inclusion relations, the combinatorial metaplex couples vertex-level intrinsic properties directly into the inclusion of HoIs.

A central contribution of the paper is the formal construction of the concentration map presented in Section \ref{1sec:CM framework}. Starting from a vertex concentration assignment map $a:V\to\mathbb{Q}^+$, we extend this assignment to a weight map on the simplicial complex. The extension is defined via fractional contribution maps that distribute the weight of lower-dimensional simplices among the higher-dimensional simplices that contain them. The resulting procedure defines a coherent propagation of intrinsic quantities across the simplicial hierarchy. A key structural result is the conservation relation established in Theorem \ref{1thm:facet-vertex-weight}, which shows that the total weight assigned to the vertices equals the total induced weight of the facets. This result establishes a vertex–facet conservation principle, whereby the total vertex weight is exactly preserved in the induced facet weights.

Within this framework, higher-dimensional simplices are included only when the total weight of their boundary simplices exceeds a prescribed threshold. This rule yields a deterministic procedure for constructing simplicial complexes from weighted vertex data. Unlike clique-based constructions, which include every combinatorially closed simplex, the CM includes only those supported by sufficient boundary weight. This distinction separates the true HoIs captured by the CM from the HoIs that arise solely through combinatorial closure in the clique complex.

Using the induced weights on simplices and the facet-mediated adjacency between simplices, we develop corresponding centrality indices for the $q$-simplices, $0 \le q \le \dim(\Delta)-1$, in Section \ref{1sec:centralities in CM}. In particular, we introduce simplicial and weighted degree centralities, together with a one-parameter family $D^{\alpha}$ that interpolates between the simplicial and weighted degrees. Weighted walks and shortest-path distances between simplices then give rise to the closeness and harmonic centrality indices $CC^{\alpha}$ and $HC^{\alpha}$, which quantify the importance of simplices in the CM.

To illustrate the proposed framework, we construct a CM from the Tuesday Lake food-web dataset using vertex-level biomass values, as described in Section \ref{1sec:Illustration of CM}. This example demonstrates how the threshold rule determines the inclusion of true HoIs and how the concentration map $a$ subsequently induces weight on these simplices. A comparison with the corresponding clique complex shows that the CM identifies a reduced set of HoIs supported by sufficient boundary weight, thereby filtering out simplices that arise solely from combinatorial closure. This highlights the role of intrinsic vertex properties in governing the inclusion and structure of higher-order networks.

The evaluation of the proposed degree, closeness, and harmonic centrality indices in the illustrative example demonstrates how the weight-driven simplicial structure influences the relative importance of simplices in the CM. In particular, the results show that the tuning parameter regulates the balance between the simplicial topology and the vertex weights in determining centrality values. As the parameter varies, the centralities transition between behaviour governed primarily by the simplicial structure and behaviour influenced more strongly by the vertex weights. Together, these observations illustrate how the combinatorial metaplex provides a principled framework for constructing and analysing higher-order networks from weighted vertex data.

\subsection*{Declaration of competing interest}

The authors have no conflicts of interest to declare that are relevant to the content of this article.

\subsection*{AI-Assistance Disclosure}

The authors used an artificial intelligence language model (ChatGPT, developed by OpenAI) for limited assistance in rephrasing, language polishing, and basic grammatical correction of certain sentences. In addition, ChatGPT was used to assist in generating and refining computational code for implementing illustrative examples of the proposed framework using Google Colab. All mathematical results, algorithms, theoretical developments, interpretations, and scientific conclusions presented in this manuscript were entirely authored, verified, and approved by the authors.

\subsection*{Funding}

This research did not receive any specific grant from funding agencies in the public, commercial, or not-for-profit sectors.
\begin{sloppypar}
 \printbibliography   
\end{sloppypar}
\end{document}